\documentclass[11pt]{article}
\usepackage{amsfonts, amsmath, amssymb, latexsym, eucal, amscd,cite} 
\newtheorem{theorem}{Theorem}[section]
\newtheorem{lemma}[theorem]{Lemma}
\newtheorem{proposition}[theorem]{Proposition}
\newtheorem{corollary}[theorem]{Corollary}
\newtheorem{exAux}[theorem]{Example}
\newenvironment{example}{\begin{exAux} \rm}{\end{exAux}}
\newtheorem{Def}[theorem]{Definition}
\newenvironment{definition}{\begin{Def} \rm}{\end{Def}}
\newtheorem{Note}[theorem]{Note}
\newenvironment{note}{\begin{Note} \rm}{\end{Note}}
\newtheorem{Problem}[theorem]{Problem}

\newtheorem{Rem}[theorem]{Remark}

\newtheorem{Not}[theorem]{Notation}

\newtheorem{Conj}[theorem]{Conjecture}

\newtheorem{Ass}[theorem]{Assumption}

\newenvironment{proof}{\medskip\noindent{\bf Proof.\ }}{\qed\medskip}
\newenvironment{proofof}[1]{\medskip\noindent{\bf Proof  of {#1}.\ 
}}{\qed\medskip}
\newcommand{\qed}{\hfill\mbox{$\Box$\qquad\qquad}}

\newcommand{\F}{\mathbb{F}}
\newcommand{\Mat}{\text{\rm Mat}}

\renewcommand{\b}[1]{( #1 )}

\newcommand{\tr}{\text{\rm tr}}
\newcommand{\vphi}{\varphi}
\renewcommand{\th}{\theta}

\newcommand{\gen}[1]{\langle #1 \rangle}

\newif\ifDRAFT

\begin{document}

\noindent
{\bf \huge Self-dual Leonard pairs}

\bigskip
\noindent
{\bf Abstract}
\\
Let $\F$ denote a field and let $V$ denote a vector space
over $\F$ with finite positive dimension.
Consider a pair $A,A^*$ of diagonalizable $\F$-linear maps on $V$,
each of which acts on an eigenbasis for the other one  in an irreducible tridiagonal fashion.
Such a pair is called a  Leonard pair.
We consider the self-dual case in which there exists an automorphism
of the endomorphism algebra of $V$ that swaps $A$ and $A^*$. 
Such an automorphism is unique, and called the duality $A \leftrightarrow A^*$.
In the present paper we give a comprehensive description of this duality.
In particular, we display an invertible $\F$-linear map $T$ on $V$
such that the map $X \mapsto T X T^{-1}$ is the duality $A \leftrightarrow A^*$.
We express $T$ as a polynomial in $A$ and $A^*$.
We describe how $T$ acts on $4$ flags, $12$ decompositions, and 24 bases for $V$.



\medskip\noindent
Kazumasa Nomura\footnote{\em Email: knomura@pop11.odn.ne.jp},
\ 
Paul Terwilliger\footnote{\em Email: terwilli@math.wisc.edu}

\medskip
\noindent
Tokyo Medical and Dental University,
Ichikawa, 272-0827, Japan
\\
\noindent
Department of Mathematics,
University of Wisconsin,
Madison,  WI53706, USA\\

\section{Introduction}
\label{sec:intro}

Let $\F$ denote a field and let $V$ denote a vector space over $\F$ with
finite positive dimension. 
We consider a pair $A, A^*$ of diagonalizable $\F$-linear maps on $V$, 
each of which acts on an eigenbasis for the other one
in an irreducible tridiagonal fashion. 
Such a pair is called a Leonard pair (see \cite[Definition 1.1]{T:Leonard}).
The Leonard pair $A, A^*$ is said to be self-dual whenever 
there exists an automorphism of the endomorphism algebra of $V$
that swaps $A$ and $A^*$.
In this case such an automorphism is unique, and called the duality $A \leftrightarrow A^*$.

The literature contains many examples of self-dual Leonard pairs.
For instance
(i)
the Leonard pair associated with an irreducible module for the Terwilliger algebra 
of the hypercube (see \cite[Corollaries 6.8, 8.5]{Go});
(ii)
a Leonard pair of Krawtchouk type  (see \cite[Definition 6.1]{NT:Krawt});
(iii)
the Leonard pair associated with an irreducible module for the Terwilliger algebra
of a distance-regular graph that has a spin model in the Bose-Mesner algebra
(see \cite[Theorem]{C:thin}, \cite[Theorems 4.1,  5.5]{CN:hyper});
(iv)
an appropriately normalized totally bipartite Leonard pair 
(see \cite[Lemma 14.8]{NT:bipartite});
(v)
the Leonard pair consisting of
any two of a modular Leonard triple $A,B,C$ (see \cite[Definition 1.4]{Cur});
(vi)
the Leonard pair consisting of a pair of opposite generators for
the $q$-tetrahedron algebra,
acting on an evaluation module (see \cite[Proposition 9.2]{IRT}).
The example (i) is a special case of (ii),
and the examples (iii), (iv) are special cases of (v).

Let $A, A^*$ denote a Leonard pair on $V$. 
We can determine whether $A, A^*$ is self-dual in the following way.
By \cite[Lemma 1.3]{T:Leonard} each eigenspace of $A$, $A^*$ has dimension one.
Let $\{\th_i\}_{i=0}^d$ denote an ordering of the eigenvalues of $A$.
For $0 \leq i \leq d$ let $v_i$ denote a $\th_i$-eigenvector for $A$.
The ordering $\{\th_i\}_{i=0}^d$ is said to be standard whenever
$A^*$ acts on the basis $\{v_i\}_{i=0}^d$ in an irreducible tridiagonal fashion.
If the ordering $\{\th_i\}_{i=0}^d$ is standard then the ordering $\{\th_{d-i}\}_{i=0}^d$
is also standard, and no further ordering is standard.
Similar comments apply to $A^*$.
Let $\{\th_i\}_{i=0}^d$ denote a standard
ordering of the eigenvalues of $A$.
Then $A,A^*$ is self-dual if and only if $\{\th_i\}_{i=0}^d$ is a standard
ordering of the eigenvalues of $A^*$ (see \cite[Proposition 8.7]{NT:affine}).

For a given self-dual Leonard pair,
it is not obvious what is the corresponding duality.
The purpose of this paper is to describe this duality.
Our description is summarized as follows.
Let $A, A^*$ denote a self-dual Leonard pair on $V$, 
and let $\{\th_i\}_{i=0}^d$ denote a standard ordering of the eigenvalues of $A$.
By construction $\{\th_i\}_{i=0}^d$ is a standard ordering of 
the eigenvalues of $A^*$. 
For $0 \leq i \leq d$ let $E_i : V \to V$ (resp.\ $E^*_i : V \to V$)
denote the projection 
onto the eigenspace of $A$ (resp.\ $A^*$) corresponding to $\th_i$.
Using the projections $\{E_i\}_{i=0}^d$ and $\{E^*_i\}_{i=0}^d$ 
we define a certain $\F$-linear map $T : V \to V$. 
We show that $T$ is invertible, and the
map $X \mapsto T X T^{-1}$ is the duality $A \leftrightarrow A^*$.
In order to illuminate the nature of $T$, we show how
$T$ acts on $4$ flags, $12$ decompositions, and $24$ bases 
attached to $A, A^*$. 
Here are some details. 
By a flag on $V$ we mean a sequence $\{H_i\}_{i=0}^d$ of subspaces of $V$
such that $H_i$ has dimension $i+1$ for $0 \leq i \leq d$ and
$H_{i-1} \subseteq H_i$ for $1 \leq i \leq d$.
By a decomposition of $V$ we mean a sequence $\{V_i\}_{i=0}^d$ of one dimensional
subspaces whose direct sum is $V$.
For a decomposition $\{V_i\}_{i=0}^d$ of $V$, define 
$H_i = V_0 + V_1 + \cdots + V_i$ for $0 \leq i \leq d$.
The sequence $\{H_i\}_{i=0}^d$ is a flag on $V$, said to be  induced by $\{V_i\}_{i=0}^d$.
Two flags $\{H_i\}_{i=0}^d$ and $\{H'_i\}_{i=0}^d$ on $V$ are called  opposite 
whenever there exists a decomposition $\{V_i\}_{i=0}^d$ of $V$
that induces $\{H_i\}_{i=0}^d$ and $\{V_{d-i}\}_{i=0}^d$ induces $\{H'_i\}_{i=0}^d$. 
In this case $V_i = H_i \cap H'_{d-i}$ for $0 \leq i \leq d$. 
In particular the decomposition $\{V_i\}_{i=0}^d$ is uniquely determined by the
ordered pair $\{H_i\}_{i=0}^d$, $\{H'_i\}_{i=0}^d$;
we say that this ordered pair induces $\{V_i\}_{i=0}^d$.
For each symbol $z$ among the symbols $0, D, 0^*, D^*$
we define a flag $[z]$ on $V$ as follows.
The flag $[0]$ is induced by $\{E_i V\}_{i=0}^d$ and the flag $[D]$ is induced by
$\{E_{d-i} V\}_{i=0}^d$. 
The flag $[0^*]$ is induced by $\{E^*_i V\}_{i=0}^d$ and the
flag $[D^*]$ is induced by $\{E^*_{d-i} V\}_{i=0}^d$.
By \cite[Theorem 7.3]{T:24points}
the flags $[0]$, $[D]$, $[0^*]$, $[D^*]$ are mutually opposite. 
For distinct $z,w$ among the symbols $0,D,0^*,D^*$,
let $[z w]$ denote the decomposition of $V$ induced by $[z]$ and $[w]$.
There are $12$ choices for the ordered pair $z,w$ and this 
gives 12 decompositions of $V$.
For each decomposition, pick a nonzero vector in
each component of the decomposition. 
The resulting sequence of vectors is a basis for $V$. 
We normalize the basis in two ways that seem attractive; 
this yields two bases for each decomposition.
By this procedure we obtain 24 bases for $V$. 
We obtain the action of $T$ on each of these bases. 
As we will see, with respect to four of the bases among the 24,
the matrix representing $T$ is independent of the four bases
and its entries take a very attractive form.

The paper is organized as follows.
In Sections \ref{sec:LP}--\ref{sec:trace} we 
review some background and establish some
basic results for general Leonard pairs.
Starting in Section \ref{sec:selfdual} we consider a self-dual Leonard pair $A,A^*$.
In Section \ref{sec:main} 
we introduce the map $T$ and discuss its basic properties. 
In Sections \ref{sec:formulaT}, \ref{sec:proofmain}
we show that $T$ is invertible, and the
map $X \mapsto T X T^{-1}$ is the duality $A \leftrightarrow A^*$.
In Section \ref{sec:decomp} we use $A, A^*$ to define 4 flags and 12 decompositions.
In Section \ref{sec:actTflag} we obtain the action of  $T$ on these flags
and decompositions.
In Sections \ref{sec:24bases}, \ref{sec:rel24}
we obtain two bases from each of the 12 decompositions, and describe how these
two bases are related. 
In Section \ref{sec:Taction}
we obtain the action of $T$ on the 24 bases.
We also display  four bases among  the 24, with respect to
which the matrix representing $T$ is independent of
the bases and its entries take an attractive form.

\section{Leonard pairs}
\label{sec:LP}

We now begin our formal argument.
In this section we recall the notion of a Leonard pair.
We use the following terms.
A square matrix is said to be {\em tridiagonal} whenever
each nonzero entry lies on either the diagonal, the subdiagonal,
or the superdiagonal.
A tridiagonal matrix is said to be {\em irreducible} whenever
each entry on the subdiagonal is nonzero and each entry on
the superdiagonal is nonzero.
Let $\F$ denote a field.

\begin{definition} \cite[Definition 1.1]{T:Leonard}
\label{def:LP}    \samepage
\ifDRAFT {\rm def:LP}. \fi
Let $V$ denote a vector space over $\F$ with finite positive dimension.
By a {\em Leonard pair} on $V$ we mean an ordered pair
of $\F$-linear maps $A: V \to V$ and $A^* : V \to V$ 
that satisfy the following {\rm (i), (ii)}.
\begin{itemize}
\item[\rm (i)]
There exists a basis for $V$ with respect to which the matrix representing $A$
is irreducible tridiagonal and the matrix representing $A^*$ is diagonal.
\item[\rm (ii)]
There exists a basis for $V$ with respect to which the matrix representing $A^*$
is irreducible tridiagonal and the matrix representing $A$ is diagonal.
\end{itemize}
We say that $A,A^*$ is {\em over} $\F$.
\end{definition}

\begin{note}
According to a common notational convention,
for a matrix $A$ its conjugate-transpose is denoted by $A^*$.
We are not using this convention.
In a Leonard pair $A,A^*$ the linear maps $A,A^*$ are arbitrary subject to (i) and (ii) above.
\end{note}

We refer the reader to \cite{T:survey} for background on Leonard pairs

\begin{note}
Assume that $A,A^*$ is a Leonard pair on $V$.
Then $A^*,A$ is a Leonard pair on $V$.
\end{note}

For the rest of this paper,
let $V$ denote a vector space over $\F$ with finite positive dimension.
Let $\text{\rm End}(V)$ denote the $\F$-algebra consisting of the
$\F$-linear maps from $V$ to $V$.
The algebra $\text{\rm End}(V)$ is called the {\em endomorphism algebra} of $V$.

\begin{lemma} \cite[Corollary 5.6]{T:survey}
\label{lem:generate}    \samepage
\ifDRAFT {\rm lem:generate}. \fi
Let $A,A^*$ denote a Leonard pair on $V$.
Then $A$, $A^*$ together generate the algebra $\text{\rm End}(V)$.
\end{lemma}

We recall the notion of an isomorphism for Leonard pairs.
Let $A,A^*$ denote a Leonard pair on $V$.
Let $V'$ denote a vector space over $\F$ with finite positive dimension,
and let $A', A^{*\prime}$ denote a Leonard pair on $V'$.
By an {\em isomorphism of Leonard pairs} from $A,A^*$ to $A', A^{*\prime}$
we mean an isomorphism of $\F$-algebras from $\text{\rm End}(V)$
to $\text{\rm End}(V')$ that sends $A \mapsto A'$ and $A^* \mapsto A^{*\prime}$.
The Leonard pairs $A,A^*$ and $A', A^{*\prime}$ are said to be {\em isomorphic}
whenever there exists an isomorphism of Leonard pairs from $A,A^*$
to $A', A^{*\prime}$.
In this case,  the isomorphism involved is unique by Lemma \ref{lem:generate}.
An isomorphism of Leonard pairs can be seen from the following point of view.
By the Skolem-Noether theorem (see \cite[Corollary 7.125]{Rot}),
a map $\sigma : \text{\rm End}(V) \to \text{\rm End}(V')$ is an $\F$-algebra
isomorphism if and only if there exists an $\F$-linear bijection $K : V \to V'$
such that $X^\sigma = K X K^{-1}$ for all $X \in \text{\rm End}(V)$.
In this case, we say that {\em $K$ gives $\sigma$}.
Assume that $K$ gives $\sigma$. 
Then an $\F$-linear map $\widetilde{K} : V \to V'$ gives $\sigma$ 
if and only if there exists $0 \not=\alpha \in \F$ such that $\widetilde{K}=\alpha K$.

\begin{definition}   \label{def:selfdual}   \samepage
\ifDRAFT {\rm def:selfdual}. \fi
A Leonard pair $A,A^*$  is said to be {\em self-dual}
whenever $A, A^*$ is isomorphic to $A^*, A$.
\end{definition}

Let $A,A^*$ denote a self-dual Leonard pair on $V$.
For an automorphism $\sigma$ of  $\text{\rm End}(V)$
the following are equivalent:
\begin{itemize}    \samepage
\item[\rm (i)]
$\sigma$ is an isomorphism of Leonard pairs from $A,A^*$ to $A^*,A$;
\item[\rm (ii)]
$\sigma$ is an isomorphism of Leonard pairs from $A^*,A$ to $A,A^*$.
\end{itemize}
There exists a unique automorphism $\sigma$ of $\text{\rm End}(V)$
that satisfies (i), (ii).

\begin{definition}    \label{def:duality}    \samepage
\ifDRAFT {\rm def:duality}. \fi
Let $A,A^*$ denote a self-dual Leonard pair on $V$.
By the {\em duality $A \leftrightarrow A^*$} we mean the
automorphism $\sigma$ of $\text{\rm End}(V)$ that satisfies (i), (ii) above.
\end{definition}

\section{Leonard systems}
\label{sec:LS}

When working with a Leonard pair, it is convenient to consider a closely related object
called a Leonard system \cite{T:Leonard}.
Before we define a Leonard system,
we recall a few concepts from linear algebra.

We denote by $I$ the identity element of $\text{\rm End}(V)$.
For $A \in \text{\rm End}(V)$ let $\gen{A}$ denote the subalgebra
of $\text{\rm End}(V)$ generated by $A$.
For an integer $d \geq 0$
let $\Mat_{d+1}(\F)$ denote the $\F$-algebra consisting of
the $d+1$ by $d+1$ matrices that have all entries in $\F$.
We index the rows and columns by $0,1,\ldots,d$.
Let $\{v_i\}_{i=0}^d$ denote a basis for $V$.
For $X \in \text{\rm End}(V)$ and $Y \in \Mat_{d+1}(\F)$,
we say that {\em $Y$ represents $X$ with respect to $\{v_i\}_{i=0}^d$} whenever
$X v_j = \sum_{i=0}^d Y_{i,j} v_i$ for $0 \leq j \leq d$.
Let $A \in \text{End}(V)$.
For $\th \in \F$ define 
$V(\th) = \{ v \in V \,|\, A v = \th v\}$.
Observe that $V(\th)$ is a subspace of $V$.
The scalar $\th$ is called an {\em eigenvalue} of $A$ whenever $V(\th) \neq 0$.
In this case, $V(\th)$ is called the {\em eigenspace} of $A$ corresponding to $\th$.
We say that
$A$ is {\em diagonalizable}  whenever $V$ is spanned by the eigenspaces of $A$.
We say that $A$ is {\em multiplicity-free}
whenever $A$ is diagonalizable, and each eigenspace of $A$ has dimension one.
Assume that $A$ is multiplicity-free, and let
$\{V_i\}_{i=0}^d$ denote an ordering of the eigenspaces of $A$.
Then $\{V_i\}_{i=0}^d$ is a decomposition of $V$.
For $0 \leq i \leq d$ let $\th_i$ denote the eigenvalue of $A$ corresponding to $V_i$.
For $0 \leq i \leq d$ define $E_i \in \text{End}(V)$ such that
$(E_i - I)V_i = 0$ and $E_i V_j = 0$ if $j \neq i$ $(0   \leq j \leq d)$.
Thus $E_i$ is the projection onto $V_i$.
Observe that
(i) $V_i = E_i V$ $(0 \leq i \leq d)$;
(ii) $E_i E_j = \delta_{i,j} E_i$ $(0 \leq i,j \leq d)$;
(iii) $I = \sum_{i=0}^d E_i$;
(iv) $A = \sum_{i=0}^d \th_i E_i$.
Also 
\begin{align}
   E_i &= \prod_{\begin{smallmatrix} 0 \leq j \leq d \\ j \neq i \end{smallmatrix}   }
           \frac{A-\th_j I}{\th_i - \th_j}    
   \qquad\qquad  (0 \leq i \leq d).                               \label{eq:Ei}
\end{align}
We call $E_i$ the {\em primitive idempotent} of $A$ for $\th_i$ $(0 \leq i \leq d)$.
Observe that $\{A^i\}_{i=0}^d$ is a basis for the $\F$-vector space $\gen{A}$,
and $\prod_{i=0}^d (A-\th_i I) =0$.
Also observe that $\{E_i\}_{i=0}^d$ is a basis for the
$\F$-vector space $\gen{A}$.

Let $A,A^*$ denote a Leonard pair on $V$.
By \cite[Lemma 1.3]{T:Leonard} each of $A$, $A^*$ is multiplicity-free.
Let $\{E_i\}_{i=0}^d$ denote an ordering of the primitive idempotents of $A$.
For $0 \leq i \leq d$ pick  $0 \neq v_i \in E_i V$. 
Then $\{v_i\}_{i=0}^d$ is a basis for $V$.
The ordering $\{E_i\}_{i=0}^d$ is said to be {\em standard}
whenever the basis $\{v_i\}_{i=0}^d$ satisfies Definition \ref{def:LP}(ii).
A standard ordering of the primitive idempotents of $A^*$ is similarly defined.

\begin{definition}    \label{def:LS}   \samepage
\ifDRAFT {\rm def:LS}. \fi
By a {\em Leonard system} on $V$  we mean a sequence
\begin{equation}
  \Phi = (A; \{E_i\}_{i=0}^d; A^*; \{E^*_i\}_{i=0}^d)     \label{eq:Phi}
\end{equation}
of elements in $\text{\rm End}(V)$
that satisfy the following (i)--(iii):
\begin{itemize}
\item[\rm (i)]
$A,A^*$ is a Leonard pair on $V$;
\item[\rm (ii)]
$\{E_i\}_{i=0}^d$ is a standard ordering of the primitive idempotents of $A$;
\item[\rm (iii)]
$\{E^*_i\}_{i=0}^d$ is a standard ordering of the primitive idempotents of $A^*$.
\end{itemize}
We say that $\Phi$ is {\em over} $\F$.
\end{definition}
Referring to Definition \ref{def:LS}, the Leonard pair $A, A^*$ from part (i)
is said to be {\it associated with $\Phi$}.

We recall the notion of an isomorphism for Leonard systems.
Consider the Leonard system \eqref{eq:Phi}.
Let $V'$ denote a vector space over $\F$ with dimension $d+1$.
For an $\F$-algebra isomorphism $\sigma : \text{\rm End}(V) \to \text{\rm End}(V')$
define 
\[
  \Phi^\sigma = (A^\sigma; \{E^\sigma_i\}_{i=0}^d; (A^*)^\sigma; \{(E^*_i)^\sigma\}_{i=0}^d).
\]
Then $\Phi^\sigma$ is a Leonard system on $V'$.
Let $\Phi'$ denote a Leonard system on $V'$.
By an {\em isomorphism of Leonard systems} from $\Phi$ to $\Phi'$
we mean an $\F$-algebra isomorphism $\sigma : \text{\rm End}(V) \to \text{\rm End}(V')$
such that $\Phi' = \Phi^\sigma$.
The Leonard systems $\Phi$ and $\Phi'$ are said to be {\em isomorphic}
whenever there exists an isomorphism of Leonard systems from $\Phi$ to $\Phi'$.
In this case, the isomorphism involved is unique.

Consider a Leonard system $\Phi = (A; \{E_i\}_{i=0}^d; A^*; \{E^*_i\}_{i=0}^d)$ over $\F$.
For $0 \leq i \leq d$ let $\th_i$ (resp.\ $\{\th^*_i\}_{i=0}^d$)
denote the eigenvalue of $A$ (resp.\ $A^*$) corresponding to $E_i$ (resp.\ $E^*_i$).
We call $\{\th_i\}_{i=0}^d$ (resp.\ $\{\th^*_i\}_{i=0}^d$)
the {\em eigenvalue sequence}  (resp.\ {\em dual eigenvalue sequence}) of $\Phi$.
Note that $\{\th_i\}_{i=0}^d$ are mutually distinct and contained in $\F$.
Similarly $\{\th^*_i\}_{i=0}^d$ are mutually distinct and contained in $\F$.

Consider a Leonard system $\Phi = (A; \{E_i\}_{i=0}^d; A^*; \{E^*_i\}_{i=0}^d)$ over $\F$.
Then each of the following is a Leonard system over $\F$:
\begin{align*}
  \Phi^* &= (A^*; \{E^*_i\}_{i=0}^d; A; \{E_i\}_{i=0}^d),
\\
  \Phi^{\downarrow} &= (A; \{E_i\}_{i=0}^d; A^*; \{E^*_{d-i}\}_{i=0}^d),
\\
  \Phi^{\Downarrow} &= (A; \{E_{d-i}\}_{i=0}^d; A^*; \{E^*_i\}_{i=0}^d).
\end{align*}
Viewing $*$, $\downarrow$, $\Downarrow$ as permutations on the set of
all Leonard systems over $\F$,
\begin{align} 
        *^2 \,=\, \downarrow^2 \,&=\, \Downarrow^2 = 1,      \label{eq:rel1}
\\
  \Downarrow * \,=\, * \downarrow, \qquad
  \downarrow * &\,=\, * \Downarrow, \qquad
  \downarrow \Downarrow \,=\, \Downarrow \downarrow.    \label{eq:rel2}
\end{align}
The group generated by the symbols $*$, $\downarrow$,  $\Downarrow$ 
subject to the relations \eqref{eq:rel1}, \eqref{eq:rel2} is the dihedral group
$D_4$.
Recall that $D_4$ is the group of symmetries of a square, and has $8$ elements.
The elements $*$, $\downarrow$, $\Downarrow$ induce an action of $D_4$ on the
set of all Leonard systems over $\F$.
Two Leonard systems over $\F$ will be called {\em relatives} whenever they are in the same
orbit of this $D_4$ action.

\begin{definition}   \label{def:fg}    \samepage
\ifDRAFT {\rm def:fg}. \fi
Let $\Phi$ denote a Leonard system, and let $g \in D_4$.
For any object $f$ attached to $\Phi$,
let $f^g$ denote the corresponding object attached to $\Phi^{g^{-1}}$.
\end{definition}

\begin{lemma}    \label{lem:associated}    \samepage
\ifDRAFT {\rm lem:associated}. \fi
Let $A,A^*$ denote a Leonard pair on $V$,
and let $\Phi$ denote an associated Leonard system.
Then the Leonard systems associated with $A,A^*$ are
$\Phi$, $\Phi^\downarrow$, $\Phi^\Downarrow$, $\Phi^{\downarrow\Downarrow}$.
\end{lemma}

\begin{proof}
By the comments above Definition \ref{def:LS}.
\end{proof}

\begin{definition}     \label{def:selfdualsystem}    \samepage
\ifDRAFT {\rm def:selfdualsystem}. \fi
A Leonard system $\Phi$ is said to be {\em self-dual}
whenever $\Phi$ is isomorphic to $\Phi^*$.
\end{definition}

Let $\Phi$ denote a self-dual Leonard system on $V$.
For an automorphism $\sigma$ of $\text{\rm End}(V)$
the following are equivalent:
\begin{itemize}    \samepage
\item[\rm (i)]
$\sigma$ is an isomorphism of Leonard systems from $\Phi$ to $\Phi^*$;
\item[\rm (ii)]
$\sigma$ is an isomorphism of Leonard systems from $\Phi^*$ to $\Phi$.
\end{itemize}
There exists a unique automorphism $\sigma$ of $\text{\rm End}(V)$
that satisfies (i), (ii).

\begin{definition}     \label{def:duality2}    \samepage
\ifDRAFT {\rm def:duality2}. \fi
Let $\Phi$ denote a self-dual Leonard system on $V$.
By the {\em duality $\Phi \leftrightarrow \Phi^*$} we mean
the automorphism of $\text{\rm End}(V)$ that satisfies (i), (ii) above.
\end{definition}

\section{Antiautomorphisms and bilinear forms}
\label{sec:anti}

In this section we recall a few notions from the theory of Leonard pairs.
Let $\cal A$ denote an $\F$-algebra.
By an {\em antiautomorphism} of $\cal A$ we mean
an $\F$-linear bijection $\xi : {\cal A} \to {\cal A}$
such that $(XY)^\xi = Y^\xi X^\xi$ for all $X$, $Y \in {\cal A}$.

\begin{lemma}   \cite[Theorem 6.1]{T:survey}
\label{lem:dagger}     \samepage 
\ifDRAFT {\rm lem:dagger}. \fi
Let $A,A^*$ denote a Leonard pair on $V$.
Then there exists a unique antiautomorphism $\dagger$ of $\text{\rm End}(V)$
such that $A^\dagger = A$ and $(A^*)^\dagger = A^*$.
Moreover $(X^\dagger)^\dagger=X$ for all $X \in \text{\rm End}(V)$.
\end{lemma}

\begin{lemma}  \cite[Lemma 6.3]{T:survey}
\label{lem:dagger2}     \samepage 
\ifDRAFT {\rm lem:dagger2}. \fi
Let $(A; \{E_i\}_{i=0}^d; A^*; \{E^*_i\}_{i=0}^d)$ denote a Leonard system on $V$.
Then the following hold.
\begin{itemize}
\item[\rm (i)]
We have $X^\dagger = X$ for all $X \in \gen{A}$.
In particular, $E_i^\dagger = E_i$ for $0 \leq i \leq d$.
\item[\rm (ii)]
We have $X^\dagger = X$ for all $X \in \gen{A^*}$.
In particular, $(E^*_i)^\dagger = E^*_i$ for $0 \leq i \leq d$.
\end{itemize}
\end{lemma}

By a {\em bilinear form on $V$} we mean a map
$\b{\; ,\,} : V \times V \to \F$ that satisfies the following
four conditions for all $u,v,w \in V$ and for all $\alpha \in \F$:
(i) $\b{u+v,w}=\b{u,w}+\b{v,w}$;
(ii) $\b{\alpha u,v}=\alpha \b{u,v}$;
(iii) $\b{u,v+w}=\b{u,v}+\b{u,w}$;
(iv) $\b{u,\alpha v}=\alpha \b{u,v}$.
Let $\b{\: ,\,}$ denote a bilinear form on $V$.
This form is said to be {\em symmetric} whenever 
$\b{u,v}=\b{v,u}$ for all $u,v \in V$.
Let $\b{\;,\,}$ denote a bilinear form on $V$.
Then the following are equivalent:
(i) there exists a nonzero $u \in V$ such that $\b{u,v}=0$ for all $v\in V$;
(ii) there exists a nonzero $v\in V$ such that $\b{u,v}=0$ for all $u\in V$.
The form $\b{\: ,\,}$ is said to be {\em degenerate} whenever (i), (ii) hold
and {\em nondegenerate} otherwise.
Let $\xi$ denote an antiautomorphism of $\text{\rm End}(V)$.
Then there exists a nonzero bilinear form $\b{\;,\,}$ on $V$
such that $\b{Xu,v}=\b{u,X^\xi v}$ for all $u,v \in V$ and 
for all $X \in \text{\rm End}(V)$.
The form is unique up to multiplication by a nonzero scalar in $\F$.
The form is nondegenerate.
We refer to this form as the {\em bilinear form on $V$ associated with $\xi$}.
This form is not symmetric in general.

Let $A, A^*$ denote a Leonard pair on $V$. 
Recall the antiautomorphism $\dagger$ of $\text{\rm End}(V)$ from Lemmma \ref{lem:dagger}.
Let $\b{\; , \,}$ denote the bilinear form on $V$ associated with $\dagger$.
By \cite[Corollary 15.4]{T:qRacah} the bilinear form  $\b{\; , \,}$ is symmetric.
By construction, for $X \in \text{\rm End}(V)$ we have
\[
  \b{Xu,v}=\b{u,X^\dagger v}  \qquad\qquad\qquad (u,v \in V).
\]
In particular,
\[
  \b{Au,v} = \b{u,Av},  \qquad\quad
  \b{A^* u,v} = \b{u, A^* v}  \qquad\qquad\qquad (u,v \in V).
\]

\section{The split decomposition and the parameter array}
\label{sec:parray}

Let $\Phi = (A; \{E_i\}_{i=0}^d; A^*; \{E^*_i\}_{i=0}^d)$ denote a Leonard system on $V$.
In this section we recall the $\Phi$-split decomposition of $V$
and the parameter array of $\Phi$.
Recall the eigenvalue sequence $\{\th_i\}_{i=0}^d$ and the dual eigenvalue sequence
$\{\th^*_i\}_{i=0}^d$ of $\Phi$.
Let $x$ denote an indeterminate, and
let $\F[x]$ denote the $\F$-algebra consisting of the polynomials in $x$
that have all coefficients in $\F$.

\begin{definition}    \cite[Definition 4.3]{T:Leonard}
\label{def:tau}    \samepage
\ifDRAFT {\rm def:tau}. \fi
For $0 \leq i \leq d$ we define some polynomials in $\F[x]$:
\begin{align*}
 \tau_i &= (x-\th_0)(x - \th_1) \cdots (x-\th_{i-1}),               \\
 \eta_i &= (x - \th_d)(x - \th_{d-1}) \cdots (x-\th_{d-i+1}),     \\
 \tau^*_i &= (x - \th^*_0)(x - \th^*_1) \cdots (x - \th^*_{i-1}),   \\
 \eta^*_i &= (x - \th^*_d)(x - \th^*_{d-1}) \cdots (x - \th^*_{d-i+1}).
\end{align*}
\end{definition}

For $0 \leq i \leq d$ define
\begin{equation}
  U_i = (E^*_0 V + \cdots + E^*_i V) \cap (E_i V + \cdots + E_d V).     \label{eq:Ui}
\end{equation}
By \cite[Theorem 20.7]{T:survey} the sequence $\{U_i\}_{i=0}^d$ is a decomposition of $V$.
This decomposition is called the {\em $\Phi$-split decomposition} of $V$.
By \cite[Lemma 20.9]{T:survey},
\begin{align*}
   (A - \th_i I) U_i &= U_{i+1} \qquad (0 \leq i \leq d-1), & (A-\th_d I) U_d &= 0, 
\\
  (A^* - \th^*_i I) U_i &= U_{i-1} \qquad (1 \leq i \leq d), & (A^*-\th^*_0 I) U_0 &=0. 
\end{align*}
For $0 \leq i \leq d$,
\begin{align*}
  \tau_i(A) U_0 &= U_i,  &
  \eta^*_i(A^*) U_d &= U_{d-i}.   
\end{align*}
Pick a nonzero $v \in E^*_0V$.
For $0 \leq i \leq d$ define 
$u_i = \tau_i(A) v$.
Then $0 \neq u_i \in U_i$ for $0 \leq i \leq d$.
Moreover, the vectors $\{u_i\}_{i=0}^d$ form a basis for $V$.
We call $\{u_i\}_{i=0}^d$ a {\em $\Phi$-split basis} for $V$.
With respect to a $\Phi$-split basis, the matrices representing $A$ and $A^*$ are
\[
 A :
  \begin{pmatrix}
    \th_0 &      &    & & & \text{\bf 0}                  \\
    1 & \th_1    \\
         & 1  & \th_2  \\
         &      &  \cdot & \cdot \\
         &       &         & \cdot & \cdot   \\
     \text{\bf 0}   &        &          &       & 1 & \th_d   \\
  \end{pmatrix},
\qquad
 A^* :
  \begin{pmatrix}
    \th^*_0 & \vphi_1     &    & & & \text{\bf 0}                  \\
        & \th^*_1 &\vphi_2    \\
         &   & \th^*_2 & \cdot \\
         &      &  & \cdot & \cdot\\
         &       &         &  & \cdot & \vphi_d   \\
     \text{\bf 0}   &        &          &       &  & \th^*_d  \\
  \end{pmatrix},
\]
where $\{\vphi_i\}_{i=1}^d$ are nonzero scalars in $\F$.
The sequence $\{\vphi_i\}_{i=1}^d$ is uniquely determined by $\Phi$,
and called the {\em first split sequence} of $\Phi$.
Let $\{\phi_i\}_{i=1}^d$ denote the first split sequence of $\Phi^\Downarrow$.
We call $\{\phi_i\}_{i=1}^d$ the {\em second split sequence} of $\Phi$. 
By the {\em parameter array} of $\Phi$ we mean the sequence
$(\{\th_i\}_{i=0}^d; \{\th^*_i\}_{i=0}^d; \{\vphi_i\}_{i=1}^d; \{\phi_i\}_{i=1}^d)$.
By \cite[Theorem 1.9]{T:Leonard} the Leonard system $\Phi$ is determined
up to isomorphism by its parameter array.

For the rest of this section let 
\[
   (\{\th_i\}_{i=0}^d; \{\th^*_i\}_{i=0}^d; \{\vphi_i\}_{i=1}^d; \{\phi_i\}_{i=1}^d)
\]
denote the parameter array of $\Phi$.

\begin{lemma}  \cite[Theorem 1.11]{T:Leonard}
\label{lem:Phis}    \samepage
\ifDRAFT {\rm lem:Phis}. \fi
The following {\rm (i)--(iii)} hold.
\begin{itemize}
\item[\rm (i)]
The parameter array of $\Phi^*$ is 
\[
   (\{\th^*_i\}_{i=0}^d; \{\th_i\}_{i=0}^d; \{\vphi_i\}_{i=1}^d; \{\phi_{d-i+1}\}_{i=1}^d).
\]
\item[\rm (ii)]
The parameter array of $\Phi^\downarrow$ is
\[
 (\{\th_i\}_{i=0}^d; \{\th^*_{d-i}\}_{i=0}^d; \{\phi_{d-i+1}\}_{i=1}^d; \{\vphi_{d-i+1}\}_{i=1}^d).
\]
\item[\rm (iii)]
The parameter array of $\Phi^\Downarrow$ is
\[
 (\{\th_{d-i}\}_{i=0}^d; \{\th^*_i\}_{i=0}^d; \{\phi_i\}_{i=1}^d; \{\vphi_i\}_{i=1}^d).
\]
\end{itemize}
\end{lemma}

We mention some results for later use.

\begin{lemma}    \label{lem:E0Ed}    \samepage
\ifDRAFT {\rm lem:E0Ed}. \fi
We have
\begin{align}                                                         \label{eq:E0Ed}
  E_0 &= \frac{\eta_d(A)}{\eta_d(\th_0)}, &
  E_d &= \frac{\tau_d(A)}{\tau_d(\th_d)}, &
  E^*_0 &= \frac{\eta^*_d(A^*)}{\eta^*_d(\th^*_0)}, &
  E^*_d &= \frac{\tau^*_d(A^*)}{\tau^*_d (\th^*_d)}.
\end{align}
\end{lemma}

\begin{proof}
By \eqref{eq:Ei} and Definition \ref{def:tau}.
\end{proof}

\begin{lemma}    \label{lem:MMs}    \samepage
\ifDRAFT {\rm note:MMs}. \fi
For the $\F$-vector spaces $\gen{A}$ and $\gen{A^*}$,
we give three bases:
\[
 \begin{array}{c|ccc}
  \text{\rm vector space $U$} & \multicolumn{3}{c}{\text{\rm three bases for $U$}} 
\\ \hline
  \gen{A} & \{E_i\}_{i=0}^d & \{\tau_i(A)\}_{i=0}^d &  \{\eta_i(A)\}_{i=0}^d   \rule{0mm}{2.5ex}
\\
  \gen{A^*} & \{E^*_i\}_{i=0}^d & \{\tau^*_i(A^*)\}_{i=0}^d &  \{\eta^*_i(A^*)\}_{i=0}^d  \rule{0mm}{2.4ex}
 \end{array}
\]
\end{lemma}

\begin{proof}
By the comments below \eqref{eq:Ei} along with Definition \ref{def:tau}.
\end{proof}

\section{Some traces}
\label{sec:trace}

Let 
$\Phi = (A; \{E_i\}_{i=0}^d; A^*; \{E^*_i\}_{i=0}^d)$
denote a Leonard system on $V$ with parameter array
\[
   (\{\th_i\}_{i=0}^d; \{\th^*_i\}_{i=0}^d; \{\vphi_i\}_{i=1}^d; \{\phi_i\}_{i=1}^d).
\]
Later in the paper we will need some facts about $\Phi$ that involve the trace function tr.
Consider the scalars
\begin{equation}
 \tr(E_rE^*_0),\qquad \tr(E_rE^*_d), \qquad \tr(E^*_rE_0), \qquad \tr(E^*_rE_d)   \label{eq:traces}
\end{equation}
for $0 \leq r \leq d$.
By \cite[Theorem 17.12]{T:qRacah} we find that for $0 \leq r \le d$,
\begin{align}     
 \tr(E_r E^*_0) &=
  \frac{\vphi_1 \vphi_2 \cdots \vphi_r \, \phi_1 \phi_2 \cdots \phi_{d-r}}
       {\eta^*_d(\th^*_0) \tau_r(\th_r) \eta_{d-r}(\th_r)},
                      \label{eq:trErEs0}  \\
 \tr(E_r E^*_d) &=
  \frac{\phi_d \phi_{d-1} \cdots \phi_{d-r+1}\,
             \vphi_d \vphi_{d-1} \cdots \varphi_{r+1}}
       {\tau^*_d(\th^*_d) \tau_r(\th_r) \eta_{d-r}(\th_r)},
                      \label{eq:trErEsd}  \\
 \tr(E^*_r E_0) &=
  \frac{\vphi_1 \vphi_2 \cdots \vphi_r\,
            \phi_d \phi_{d-1} \cdots \phi_{r+1}}
       {\eta_d(\th_0) \tau^*_r(\th^*_r) \eta^*_{d-r}(\th^*_r)},
                      \label{eq:trEsrE0}  \\
 \tr(E^*_r E_d) &=
  \frac{\phi_1 \phi_2 \cdots \phi_r\,
            \vphi_d \vphi_{d-1} \cdots \vphi_{r+1}}
       {\tau_d(\th_d) \tau^*_r(\th^*_r) \eta^*_{d-r}(\th^*_r)}.
                      \label{eq:trEsrEd} 
\end{align}
Note that the scalars in \eqref{eq:trErEs0}--\eqref{eq:trEsrEd} are nonzero.
In particular $\tr(E_0E^*_0)$ is nonzero.
Define $\nu \in \F$ by
\begin{equation}
          \nu = \text{\rm tr} (E_0 E^*_0)^{-1}.                   \label{eq:defnu}
\end{equation}
By \cite[Lemma 9.4]{T:survey},
\begin{equation}
   \nu E_0 E^*_0 E_0 = E_0,  \qquad\qquad\qquad
   \nu E^*_0 E_0 E^*_0 = E^*_0.                         \label{eq:E0Es0E0}
\end{equation}
By \eqref{eq:trErEs0}--\eqref{eq:defnu},
\begin{align}
\nu &= \frac{\eta_d(\th_0) \eta^*_d(\th^*_0)}
                  {\phi_1 \cdots \phi_d},     
&
\nu^\downarrow &=
   \frac{\eta_d(\th_0) \tau^*_d(\th^*_d)}
          {\vphi_1 \cdots \vphi_d},                                    \label{eq:nud}
\\
   \nu^\Downarrow &= 
            \frac{\tau_d(\th_d) \eta^*_d(\th^*_0)}
                   {\vphi_1 \cdots \vphi_d},
&
   \nu^{\downarrow\Downarrow} &= 
            \frac{\tau_d(\th_d) \tau^*_d(\th^*_d)}
                   {\phi_1 \cdots \phi_d}.                                 \label{eq:nuD}
\end{align}

We mention a result for later use.
Let $\{U_i\}_{i=0}^d$ denote the $\Phi$-split decomposition of $V$.
For $0 \leq i \leq d$ define $F_i \in \text{\rm End}(V)$
such that $(F_i - I)U_i=0$ and $F_i U_j=0$ if $j \neq i$ $(0 \leq j \leq d)$.
Thus $F_i$ is the projection onto $U_i$.
Observe that (i) $U_i = F_i V$ $(0 \leq i \leq d)$;
(ii) $F_i F_j =\delta_{i,j} F_i$  $(0 \leq i,j\leq d)$;
(iii) $I = \sum_{i=0}^d F_i$.

\begin{lemma}   \cite[Corollary 7.4]{NT:unit}
\label{lem:Fi}    \samepage
\ifDRAFT {\rm lem:Fi}. \fi
For $0 \leq i \leq d$,
\begin{align}
  F_i &=  \frac{\nu \tau_i(A) E^*_0 E_0 \tau^*_i (A^*)}
                {\vphi_1 \cdots \vphi_i}.                          \label{eq:Fi}
\end{align}
\end{lemma}

\section{Self-dual Leonard pairs and systems}
\label{sec:selfdual}

Earlier we defined the concept of a self-dual  Leonard pair and system. 
In this section we make some observations about this concept.

\begin{lemma}    \label{lem:sd}    \samepage
\ifDRAFT {\rm lem:sd}. \fi
Let $A,A^*$ denote a self-dual Leonard pair on $V$,
and let $\sigma$ denote the duality $A \leftrightarrow A^*$.
Them $\sigma^2 = 1$.
\end{lemma}

\begin{proof}
By construction, $\sigma^2$ fixes each of $A$, $A^*$.
By this and Lemma \ref{lem:generate}, $\sigma^2$ fixes every element
of $\text{\rm End}(V)$. So $\sigma^2 = 1$.
\end{proof}

\begin{lemma}
Let $\Phi = (A; \{E_i\}_{i=0}^d; A^*; \{E^*_i\}_{i=0}^d)$ denote a self-dual Leonard system, 
and let $\sigma$ denote  the duality $\Phi \leftrightarrow \Phi^*$. 
Then the Leonard pair $A, A^*$ is self-dual. 
Moreover $\sigma$ is the duality $A \leftrightarrow A^*$.
\end{lemma}

\begin{proof}
By construction.
\end{proof}

\begin{lemma} \label{lem:selfdualpair}    \samepage
\ifDRAFT {\rm lem:selfdualpair}. \fi
Let $A,A^*$ denote a self-dual Leonard pair, and let $\sigma$ denote the
duality $A \leftrightarrow A^*$.
Let $\{E_i\}_{i=0}^d$ denote a standard ordering of the primitive idempotents of $A$.
Then the following {\rm (i)--(iii)} hold:
\begin{itemize}
\item[\rm (i)]
$\{E^\sigma_i\}_{i=0}^d$ is a standard ordering of the primitive idempotents of $A^*$;
\item[\rm (ii)]
the sequence $\Phi =(A; \{E_i\}_{i=0}^d; A^*; \{E^\sigma_i\}_{i=0}^d)$ is a self-dual Leonard system;
\item[\rm (iii)]
$\sigma$ is the duality $\Phi \leftrightarrow \Phi^*$.
\end{itemize}
\end{lemma}

\begin{proof}
Note that $A^\sigma = A^*$ and $(A^*)^\sigma = A$.

(i)
Let $\{E^*_i\}_{i=0}^d$ denote a standard ordering of the primitive idempotents for $A^*$,
and consider the Leonard system
\[
   \Phi' = (A; \{E_i\}_{i=0}^d; A^*; \{E^*_i\}_{i=0}^d).
\]
We have $(\Phi') ^\sigma = (A^* ; \{E^\sigma_i\}_{i=0}^d; A; \{(E^*_i)^\sigma\}_{i=0}^d)$.
The result follows since $(\Phi')^\sigma$ is a Leonard system.

(ii), (iii)
By (i) above and the construction,
$\Phi$ is a Leonard system.
Applying $\sigma$ to $\Phi$ and using Lemma \ref{lem:sd}, we obtain
\[
      \Phi^\sigma = (A^*; \{E^\sigma_i\}_{i=0}^d; A; \{E_i\}_{i=0}^d) = \Phi^*.
\]
The result follows.
\end{proof}

The self-dual Leonard systems are characterized as follows.

\begin{lemma}  \cite[Proposition 8.7]{NT:affine}
\label{lem:selfdualparam}    \samepage
\ifDRAFT {\rm lem:selfdualparam}. \fi
Let $\Phi$ denote a Leonard system over $\F$ with
parameter array 
$(\{\th_i\}_{i=0}^d; \{\th^*_i\}_{i=0}^d; \{\vphi_i\}_{i=1}^d; \{\phi_i\}_{i=1}^d)$.
Then $\Phi$ is self-dual if and only if
\begin{align}
   \th_i &= \th^*_i  && (0 \leq i \leq d).          \label{eq:thths}
\end{align}
In this case
\begin{align}
  \phi_i &= \phi_{d-i+1}  &&  (1 \leq i \leq d).    \label{eq:phi}
\end{align}
\end{lemma}

Let $\Phi = (A; \{E_i\}_{i=0}^d; A^*; \{E^*_i\}_{i=0}^d)$
denote a self-dual Leonard system on $V$,
and let $\sigma$ denote the duality $\Phi \leftrightarrow \Phi^*$.
Our next general goal is to describe $\sigma$.  
To do this we will display an invertible $T \in \text{\rm End}(V)$ that gives $\sigma$.

\section{The element $T$}
\label{sec:main}

For the rest of the paper, fix a Leonard system on $V$:
\begin{equation}
  \Phi = (A; \{E_i\}_{i=0}^d; A^*; \{E^*_i\}_{i=0}^d).                        \label{eq:Phi2}
\end{equation}
In this section we introduce an element $T \in \text{\rm End}(V)$;
this element will be used to describe the duality $\Phi \leftrightarrow \Phi^*$
in the self-dual case.
Let 
\[
   (\{\th_i\}_{i=0}^d; \{\th^*_i\}_{i=0}^d; \{\vphi_i\}_{i=1}^d; \{\phi_i\}_{i=1}^d).
\]
denote the parameter array of $\Phi$.
Let $\dagger$ denote the antiautomorphism of $\text{\rm End}(V)$ that
fixes each of $A$, $A^*$.
Let $\b{\; , \,}$ denote the bilinear form on $V$ associated with $\dagger$,
as discussed at the end of Section \ref{sec:anti}.

\begin{definition}    \label{def:T}     \samepage
\ifDRAFT {\rm def:T}. \fi
Define $T \in \text{\rm End}(V)$ by
\begin{equation}
   T = \sum_{i=0}^d \eta_{d-i}(A) E^*_0 E_d \tau^*_i(A^*).      \label{eq:T}
\end{equation}
\end{definition}

\begin{note}    
Sometimes it is convenient to express $T$ as a polynomial in $A, A^*$.
Evaluating \eqref{eq:T} using \eqref{eq:E0Ed} we get
\[
 T =  \sum_{i=0}^d \frac{\eta_{d-i}(A) \eta^*_d(A^*) \tau_d(A) \tau^*_i (A^*)}
                                {\tau_d(\th_d) \eta^*_d(\th^*_0)}.
\]
\end{note}

We have
\begin{align*}
 T^* &= \sum_{i=0}^d \eta^*_{d-i}(A^*) E_0 E^*_d \tau_i(A),  
\\
 T^\dagger &= \sum_{i=0}^d \tau^*_i(A^*) E_d E^*_0 \eta_{d-i}(A), 
\\
 (T^*)^\dagger &= \sum_{i=0}^d \tau_i(A) E^*_d E_0 \eta^*_{d-i}(A^*).  
\end{align*}

We now state our first main result.

\begin{theorem}    \label{thm:main}    \samepage
\ifDRAFT {\rm thm:main}. \fi
Assume that $\Phi$ is self-dual.
Then the elements $T$, $T^*$, $T^\dagger$, $(T^*)^\dagger$
are equal and this common element gives the duality $\Phi \leftrightarrow \Phi^*$.
\end{theorem}

Our proof of Theorem \ref{thm:main} is contained in Section \ref{sec:proofmain}.

\section{Some products}
\label{sec:formulaT}

We continue to discuss the Leonard system $\Phi$ from \eqref{eq:Phi2}.
Recall the element $T$ from Definition \ref{def:T}.
In this section we consider the elements $T$, $T^*$, $T^\dagger$, $(T^*)^\dagger$.
We obtain formulas for the products of these elements with the
elements $E_0$, $E^*_0$.
These formulas are used to show that $T=T^* = T^\dagger$ 
in our proof of Theorem \ref{thm:main}.

\begin{lemma}   \label{lem:TEs0}    \samepage
\ifDRAFT {\rm lem:TEs0}. \fi
We have
\begin{align}
 T E^*_0 &= 
  \frac{\eta_d(\th_0) \vphi_1 \cdots \vphi_d}
         {\tau_d(\th_d) \eta^*_d(\th^*_0)} \,
     E_0 E^*_0,                                                      \label{eq:TEs0}
\\
 T^* E_0 &=
  \frac{\eta^*_d(\th^*_0) \vphi_1 \cdots \vphi_d}
         {\tau^*_d (\th^*_d) \eta_d (\th_0)} \,
        E^*_0 E_0,                                                     \label{eq:TsE0}
\\
 E^*_0 T^\dagger &=
     \frac{\eta_d(\th_0) \vphi_1 \cdots \vphi_d}
         {\tau_d(\th_d) \eta^*_d(\th^*_0)} \,
    E^*_0 E_0,                                                      \label{eq:Es0Td}
\\
  E_0 (T^*)^\dagger  &=
    \frac{\eta^*_d(\th^*_0) \vphi_1 \cdots \vphi_d}
         {\tau^*_d (\th^*_d) \eta_d (\th_0)} \,
  E_0 E^*_0.                                                 \label{eq:E0Tsd}
\end{align}
\end{lemma}

\begin{proof}
We first show \eqref{eq:TEs0}.
In \eqref{eq:T}, multiply each side on the right by $E^*_0$.
Simplify the result using
$\tau^*_i (A^*) E^*_0 = \tau^*_i(\th^*_0) E^*_0$
and $\tau^*_i(\th^*_0) = \delta_{i,0}$ $(0 \leq i \leq d)$ to get
\[
    T E^*_0 = \eta_d(A) E^*_0 E_d E^*_0.
\]
By \eqref{eq:E0Ed} we have $\eta_d(A) = \eta_d(\th_0) E_0$.
By \eqref{eq:E0Es0E0} applied to $\Phi^\Downarrow$,
$E^*_0 E_d E^*_0 = (\nu^\Downarrow)^{-1} E^*_0$.
By these comments and \eqref{eq:nuD} we obtain \eqref{eq:TEs0}.
The line \eqref{eq:TsE0} is obtained by applying \eqref{eq:TEs0} to $\Phi^*$.
The lines \eqref{eq:Es0Td} and \eqref{eq:E0Tsd} are obtained by applying $\dagger$
to \eqref{eq:TEs0} and \eqref{eq:TsE0}, respectively.
\end{proof}

\begin{lemma}  \cite[Lemma 7.1]{NT:maps}
\label{lem:tauiA}       \samepage
\ifDRAFT {\rm lem:tauiA}. \fi
For $0 \leq i,j \leq d$,
\begin{align}    
 E^*_0 \tau_i(A)\tau^*_j(A^*)E_0 &= 
  \delta_{i,j} \, \varphi_1\varphi_2\cdots\varphi_i \, E^*_0E_0.         \label{eq:Es0tauitausjE0} 
\end{align}
\end{lemma}

\begin{lemma}   \label{lem:TE0}    \samepage
\ifDRAFT {\rm lem:TE0}. \fi
We have
\begin{align}
 T E_0 &= \frac{\vphi_1 \cdots \vphi_d}{\tau_d(\th_d)} \, E^*_0 E_0,     \label{eq:TE0}
\\
 T^* E^*_0 &=  \frac{\vphi_1 \cdots \vphi_d}{\tau^*_d(\th^*_d)} \, E_0 E^*_0,     \label{eq:TsEs0}
\\
 E_0 T^\dagger &=   \frac{\vphi_1 \cdots \vphi_d}{\tau_d(\th_d)} \, E_0 E^*_0,     \label{eq:E0Td}
\\
  E^*_0 (T^*)^\dagger  &= \frac{ \vphi_1 \cdots \vphi_d}{\tau^*_d(\th^*_d)} \, E^*_0 E_0.    \label{eq:Es0Tsd}
\end{align}
\end{lemma}

\begin{proof}
In \eqref{eq:T}, multiply each side on the right by $E_0$.
Simplify the result using $E_d=\tau_d(A) / \tau_d(\th_d)$ and \eqref{eq:Es0tauitausjE0} to get \eqref{eq:TE0}.
The line \eqref{eq:TsEs0} is obtained by applying \eqref{eq:TE0} to $\Phi^*$.
The lines \eqref{eq:E0Td} and \eqref{eq:Es0Tsd} are obtained by applying $\dagger$
to \eqref{eq:TE0} and \eqref{eq:TsEs0}, respectively.
\end{proof}

\section{The proof of Theorem \ref{thm:main}}
\label{sec:proofmain}

In this section we prove Theorem \ref{thm:main}.
Recall the Leonard system $\Phi$ from \eqref{eq:Phi2} and
the element $T$ from Definition \ref{def:T}.

\begin{lemma}     \label{lem:T2}     \samepage
\ifDRAFT {\rm lem:T2}. \fi
We have
\begin{equation}
T^2 = 
  (\nu^\Downarrow)^{-1} \phi_1 \cdots \phi_d
   \sum_{j=0}^d
  \frac{\eta_j(A) E^*_0 E_d \tau^*_j(A^*)}
         {\phi_d \cdots \phi_{d-j+1}}.                              \label{eq:T2}
\end{equation}
\end{lemma}

\begin{proof}
By \eqref{eq:T}, 
\begin{equation}     
T^2=  \sum_{i=0}^d \sum_{j=0}^d 
    \eta_{d-i}(A) E^*_0 E_d \tau^*_i(A^*)
    \eta_{d-j}(A) E^*_0 E_d \tau^*_j(A^*).     \label{eq:aux1}
\end{equation}
Applying \eqref{eq:Es0tauitausjE0} to $\Phi^{\Downarrow *}$,
\[
   E_d \tau^*_i(A^*) \eta_j (A) E^*_0 = \delta_{i,j} \phi_1 \cdots \phi_i E_d E^*_0
             \qquad\qquad (0 \leq i,j \leq d).
\]
In this line, replace $j$ with $d-j$ to get
\[
  E_d \tau^*_i (A^*) \eta_{d-j}(A) E^*_0
  = \delta_{i, d-j} \phi_1 \cdots \phi_{d-j} E_d E^*_0 \qquad\qquad  (0 \leq i,j \leq d).
\]
By this and \eqref{eq:aux1},
\begin{equation}
T^2 = \sum_{j=0}^d \phi_1 \cdots \phi_{d-j}
          \eta_j(A) E^*_0 E_d E^*_0 E_d \tau^*_j (A^*).    \label{eq:aux2}
\end{equation}
Applying \eqref{eq:E0Es0E0} to $\Phi^\Downarrow$,
\[
  E^*_0 E_d E^*_0 = (\nu^\Downarrow)^{-1} E^*_0.
\]
By this and \eqref{eq:aux2} we get \eqref{eq:T2}.
\end{proof}

\begin{proposition}   \label{prop:T2}    \samepage
\ifDRAFT {\rm prop:T2}. \fi
Assume that $\Phi$ is self-dual,
Then $T$ is invertible.
Moreover, $T^2 = \lambda I$, 
where
\[
   \lambda = (\nu^\Downarrow)^{-2} \phi_1 \cdots \phi_d.
\]
\end{proposition}

\begin{proof}
By Lemma \ref{lem:selfdualparam} the sum in \eqref{eq:T2} is equal to
\[
 \sum_{j=0}^d
 \frac{\eta_j(A) E^*_0 E_d \tau^*_j(A^*)}
         {\phi_1 \cdots \phi_{j}}. 
\]
Applying \eqref{eq:Fi} to $\Phi^\Downarrow$ and using $I=\sum_{j=0}^d F_j^\Downarrow$,
we find that the above sum is equal to $(\nu^\Downarrow)^{-1} I$.
Thus $T^2 = \lambda I$.
By construction $\lambda \neq 0$ so $T$ is invertible.
\end{proof}

\begin{lemma}    \label{lem:ATTAs}    \samepage
\ifDRAFT {\rm lem:ATTAs}. \fi
Assume that $\Phi$ is self-dual.
Then
\begin{align}
 A T &= T A^*,                               
&
 A^* T &= T A.           \label{eq:ATTAs}
\end{align}
\end{lemma}

\begin{proof}
We first show $AT=TA^*$.
For $0 \leq i \leq d$ define
\[
   T_i = \eta_{d-i}(A) E^*_0 E_d \tau^*_i(A^*).
\]
The element $T_i$ is the $i$-summand in \eqref{eq:T}, so $T = \sum_{i=0}^d T_i$. 
By Definition \ref{def:tau} along with
$\prod_{\ell=0}^d (A-\th_\ell I) = 0$ and
$\prod_{\ell=0}^d (A^* - \th^*_\ell I) =0$,
\begin{align*}
 A T_0 - \th_0 T_0 &= 0,
\\
 A T_i - \th_i T_i &= T_{i-1} A^* - \th^*_{i-1} T_{i-1}    \qquad\qquad(1 \leq i \leq d),
\\
 0 &= T_d A^* - \th^*_d T_d.
\end{align*}
By these comments
\[
   A T - \sum_{i=0}^d \th_i T_i = T A^* - \sum_{i=0}^d \th^*_i T_i.
\]
By this and \eqref{eq:thths} we see that $A T = T A^*$.
In this equation, multiply each side on the left and right by $T$.
Simplify the result using Proposition \ref{prop:T2} to get $A^* T = T A$.
\end{proof}

\begin{corollary}    \label{cor:EiTTEis}   \samepage
\ifDRAFT {\rm cor:EiTTEis}. \fi
Assume that $\Phi$ is self-dual.
Then
\begin{align}
  E_i T &= T E^*_i, &  E^*_i T &= T E_i && (0 \leq i \leq d).     \label{eq:EiT}
\end{align}
\end{corollary}

\begin{proof}
By \eqref{eq:Ei} and \eqref{eq:ATTAs}.
\end{proof}

\begin{proofof}{Theorem \ref{thm:main}}
By Proposition \ref{prop:T2}, $T$ is invertible.
By \eqref{eq:ATTAs} and \eqref{eq:EiT},
$T$ gives the duality $\Phi \leftrightarrow \Phi^*$.

Next we show that $T=T^*$.
In the above statement, we replace $T$ by $T^*$ and swap the roles of $\Phi$, $\Phi^*$
to see that $T^*$ gives the duality $\Phi \leftrightarrow \Phi^*$.
Thus each of $T$ and $T^*$ gives the duality $\Phi  \leftrightarrow \Phi^*$.
By this and the comment above Definition \ref{def:selfdual},
there exists $0 \neq \zeta \in \F$ such that $T^* = \zeta T$.
We show that $\zeta = 1$.
By \eqref{eq:TsE0}, \eqref{eq:TE0} together with \eqref{eq:thths}
we find that $T^* E_0$ and $T E_0$ have the same trace.
This trace is nonzero by the comments above \eqref{eq:defnu}.
Thus $\zeta = 1$ and so $T=T^*$.

Next we show that $T=T^\dagger$.
In \eqref{eq:ATTAs} and \eqref{eq:EiT}, apply $\dagger$ to each side
and use Lemma \ref{lem:dagger2} to find that $T^\dagger$ gives the duality $\Phi \leftrightarrow \Phi^*$.
Thus each of $T$ and $T^\dagger$ gives the duality $\Phi  \leftrightarrow \Phi^*$.
By this and the comment above Definition \ref{def:selfdual},
there exists $0 \neq \zeta' \in \F$ such that $T^\dagger = \zeta' T$.
We show that $\zeta' = 1$.
By  \eqref{eq:Es0Td}, \eqref{eq:TsEs0} together with \eqref{eq:thths} and $T=T^*$,
we find that $E^*_0 T^\dagger$ and $E^*_0 T$ have the same trace.
This trace is nonzero by the comments above \eqref{eq:defnu}.
Thus $\zeta' = 1$ and so $T=T^\dagger$.

In the equation $T = T^*$, apply $\dagger$ to each side to get $T^\dagger = (T^*)^\dagger$.
We have shown that the elements $T$, $T^*$, $T^\dagger$, $(T^*)^\dagger$ are equal.
\end{proofof}

\section{Some decompositions and flags associated with a Leonard system}
\label{sec:decomp}

Consider the Leonard system $\Phi$ from \eqref{eq:Phi2}.
Recall from Section 1 the notion of a flag on $V$,
and what it means for two flags on $V$ to be opposite.
In this section we use $\Phi$ to obtain four mutually
opposite flags on $V$;
these are induced by the eigenspace decompositions of $A$ and $A^*$,
as well as the split decomposition for $\Phi$ and its relatives.
In the next section, we will describe how $T$ acts on these flags
and decompositions.

\begin{definition}          \label{def:Omega}  \samepage
For notational convenience let
$\Omega$ denote the set consisting of four symbols
$0,D,0^*,D^*$.
\end{definition}

\begin{definition}          \label{def:flags}   \samepage
\ifDRAFT {\rm def:flags}. \fi
For $z \in \Omega$ we define a flag on $V$ which we denote by $[z]$.
To define this flag we display the $i^\text{th}$ component for $0 \leq i \leq d$.
\[
\begin{array}{c|c}
    z & \text{$i^\text{th}$ component of $[z]$}  \\
  \hline
    0 & E_0V+E_1V+\cdots+E_iV                          \rule{0mm}{2.7ex}     \\
    D & E_dV+E_{d-1}V+\cdots+E_{d-i}V                \rule{0mm}{2.5ex}   \\
   0^* & E^*_0V+E^*_1V+\cdots+E^*_iV               \rule{0mm}{2.5ex}  \\
   D^* & E^*_dV+ E^*_{d-1}V+\cdots+E^*_{d-i}V    \rule{0mm}{2.5ex}  
\end{array}
\]
\end{definition}

\begin{lemma} \cite[Theorem 7.3]{T:24points}
\label{lem:opposite} \samepage
\ifDRAFT {\rm lem:opposite}. \fi
The four flags in Definition \ref{def:flags} are mutually opposite.
\end{lemma}

\begin{definition}          \label{def:decompositions}   \samepage
\ifDRAFT {\rm def:decompositions}. \fi
Let $z,w$ denote an ordered pair of distinct elements of $\Omega$.
By Lemma \ref{lem:opposite} the flags $[z]$, $[w]$ are opposite.
Let $[zw]$ denote the decomposition of $V$ induced by $[z]$, $[w]$.
\end{definition}

Let $\{V_i\}_{i=0}^d$ denote a decomposition of $V$.
By the {\em inversion} of this decomposition we mean
the decomposition $\{V_{d-i}\}_{i=0}^d$.
By \cite[Lemma 8.6]{NT:switch}
the decompositions in Definition \ref{def:decompositions}
have the following features. 
For distinct $z$, $w \in \Omega$ we have
(i) 
the decomposition $[zw]$ is the inversion of $[wz]$;
(ii) 
for $0 \leq i \leq d$ the $i^\text{th}$ component of $[zw]$ is the intersection
of the $i^\text{th}$  component of $[z]$ and the $(d-i)^\text{th}$ component of $[w]$;
(iii) 
the decomposition $[zw]$ induces $[z]$ and the inversion of $[zw]$ induces $[w]$.

\begin{example}   \label{exam:decompositions}  \samepage
\ifDRAFT {\rm exam:decompositions}. \fi
We display some of the decompositions from 
Definition \ref{def:decompositions}.
For each decomposition in the table below we give
the $i^\text{th}$ component for $0 \leq i \leq d$.
\[
\begin{array}{c|c}
 \text{decomposition} & \text{$i^\text{th}$ component}  \\
\hline
  \;\;\; [0^*D] \;\;\; & 
   \;\;\; (E^*_0V+\cdots+E^*_iV)\cap(E_iV+\cdots+E_dV)     \rule{0mm}{2.7ex}      \\
  {[D^*D]} & 
    (E^*_dV+\cdots+E^*_{d-i}V)\cap(E_iV+\cdots+E_dV)       \rule{0mm}{2.5ex}    \\
  {[0^*0]}  &
     (E^*_0V+\cdots+E^*_{i}V)\cap(E_{d-i}V+\cdots+E_{0}V)       \rule{0mm}{2.5ex}  \\
  {[D^*0]}  & 
     (E^*_dV+\cdots+E^*_{d-i}V)\cap(E_{d-i}V+\cdots+E_{0}V)       \rule{0mm}{2.5ex}   \\
  {[0D]}  &
       E_iV \\
  {[0^*D^*]} &
       E^*_iV 
\end{array}
\]
\end{example}

\section{The action of $T$ on the flags and decompositions}
\label{sec:actTflag}

Recall the Leonard system $\Phi$ from \eqref{eq:Phi2}
and the element $T$  from Definition \ref{def:T}.
In this section we describe how $T$ acts on the flags from Definition \ref{def:flags}
and the decompositions from Definition \ref{def:decompositions}.

\begin{lemma}    \label{lem:TEiV}    \samepage
\ifDRAFT {\rm lem:TEiV}. \fi
Assume that $\Phi$ is self-dual.
Then 
\[
  T E_i V = E^*_i V,  \qquad\qquad T E^*_i V = E_i V  \qquad\qquad (0 \leq i \leq d).
\]
\end{lemma}

\begin{proof}
By \eqref{eq:EiT},  $T E_i V = E^*_i TV$.
We have $T V = V$ since $T$ is invertible.
By these comments $T E_i V = E^*_i V$.
Similarly $T E^*_i V = E_i V$.
\end{proof}

For a sequence $H=\{H_i\}_{i=0}^d$ of subspaces of $V$,
let $T H$ denote the sequence $\{T H_i\}_{i=0}^d$.

\begin{proposition}    \label{prop:Tflags} \samepage
\ifDRAFT {\rm prop:Tflags}. \fi
Assume that $\Phi$ is self-dual.
Then
\begin{align*}
T [0] &= [0^*], & T[0^*] &=[0], & T[D] &=[D^*], & T[D^*] &=[D].
\end{align*}
\end{proposition}

\begin{proof}
By Definition \ref{def:flags} and Lemma \ref{lem:TEiV}.
\end{proof}

\begin{proposition}    \label{prop:Tdecomp}    \samepage
\ifDRAFT {\rm prop:Tdecomp}. \fi
Assume that $\Phi$ is self-dual.
In the table below we give some decompositions $u$ of $V$.
For each decomposition $u$ we give $T u$.
\[
\begin{array}{c|cccccc}
 u & [0^*D] & [D^*D] & [0^*0] & [D^*0] & [0D] & [0^*D^*]
\\ \hline
T u & [0D^*] & [DD^*] & [00^*] & [D0^*] & [0^*D^*] & [0D]    \rule{0mm}{2.4ex}
\end{array}
\]
\end{proposition}

\begin{proof}
First consider the case $u = [0^* D]$.
By Definition \ref{def:decompositions} the decomposition $u$ is induced by
the ordered pair of flags $[0^*]$, $[D]$.
By this and since $T$ is invertible,
the decomposition $T u$ is induced by the ordered pair of flags
$T [0^*]$, $T[D]$.
By Proposition \ref{prop:Tflags} we have $T [0^*] = [0]$ and $T[D]= [D^*]$.
By these comments and Definition \ref{def:decompositions}, $T u = [0 D^*]$.
We have shown the result for the case $u = [0^* D]$.
For the other cases the proof is similar.
\end{proof}

\section{The 24 bases}
\label{sec:24bases}

Recall the Leonard system $\Phi$ from \eqref{eq:Phi2}
and the element $T$ from Definition \ref{def:T}.
In \cite{T:24points} the second author introduced $24$ bases for $V$
on which $A$, $A^*$ act in an attractive manner.
Our next goal is to describe how $T$ acts on these bases.
In this section we define the $24$ bases and give their basic properties.

Let $v_0$, $v_d$, $v^*_0$, $v^*_d$ denote nonzero vectors in $V$ such that
\begin{equation}          \label{eq:defv0vd}
   v_0 \in E_0V,  \qquad
   v_d \in E_dV,  \qquad
   v^*_0 \in E^*_0V, \qquad
   v^*_d \in E^*_dV.
\end{equation}
We consider the decompositions from Definition \ref{def:decompositions}.

\begin{lemma}    \label{lem:0sD}    \samepage
\ifDRAFT {\rm lem:0sD}. \fi
For each row in the table below, consider the decomposition $\{U_i\}_{i=0}^d$ of $V$
in the first column.
For $0 \leq i \leq d$
the vector in the second column and third column is a basis for $U_i$.
\[
\begin{array}{c|c|c}
\text{\rm decomposition $\{U_i\}_{i=0}^d$} & \text{\rm basis for $U_i$}  
      &  \text{\rm basis for $U_i$}
\\ \hline
{[0^* D]} & \tau_i (A) v^*_0 & \eta^*_{d-i}(A^*) v_d    \rule{0mm}{2.7ex}
\\
{[D^* D]} & \tau_i (A) v^*_d & \tau^*_{d-i}(A^*) v_d    \rule{0mm}{2.5ex}
\\
{[0^* 0]} & \eta_i (A) v^*_0 & \eta^*_{d-i}(A^*) v_0    \rule{0mm}{2.5ex}
\\
{[D^* 0]} & \eta_i (A) v^*_d & \tau^*_{d-i} (A^*) v_0    \rule{0mm}{2.5ex}
\end{array}
\]
\end{lemma}

\begin{proof}
By \cite[Lemma 8.8]{NT:switch}.
\end{proof}

\begin{corollary} 
\label{cor:0sD}    \samepage
\ifDRAFT {\rm cor:0sD}. \fi
Each of the following $8$ sequences is a basis for $V$:
\begin{align}
& \{\tau_i(A) v^*_0\}_{i=0}^d,  
&& \{\tau_i(A) v^*_d\}_{i=0}^d,
&& \{\eta_i(A) v^*_0\}_{i=0}^d,
&& \{\eta_i(A) v^*_d\}_{i=0}^d,             \label{eq:tauivs0}
\\
& \{\tau^*_{d-i}(A^*) v_0\}_{i=0}^d,
&& \{\tau^*_{d-i}(A^*) v_d \}_{i=0}^d,
&& \{\eta^*_{d-i}(A^*) v_0\}_{i=0}^d,
&& \{\eta^*_{d-i}(A^*) v_d \}_{i=0}^d.           \label{eq:tausd-iv0}
\end{align}
\end{corollary}

\begin{proof}
By Lemma \ref{lem:0sD}.
\end{proof}

\begin{lemma}    \label{lem:D0s}    \samepage
\ifDRAFT {\rm lem:D0s}. \fi
For each row in the table below, consider the decomposition $\{U_i\}_{i=0}^d$ of $V$
in the first column.
For $0 \leq i \leq d$
the vector in the second column and third column is a basis for $U_i$.
\[
\begin{array}{c|c|c}
\text{\rm decomposition $\{U_i\}_{i=0}^d$} 
   & \text{\rm basis for $U_i$}  &  \text{\rm basis for $U_i$}
\\ \hline
{[D 0^*]} & \eta^*_i (A^*) v_d & \tau_{d-i}(A) v^*_0    \rule{0mm}{2.7ex}
\\
{[D D^*]} & \tau^*_i (A^*)  v_d & \tau_{d-i}(A) v^*_d   \rule{0mm}{2.5ex}
\\
{[0 0^*]} & \eta^*_i (A^*) v_0 & \eta_{d-i} (A) v^*_0   \rule{0mm}{2.5ex}
\\
{[0 D^*]} &\tau^*_i (A^*) v_0 & \eta_{d-i}(A) v^*_d   \rule{0mm}{2.5ex}
\end{array}
\]
\end{lemma}

\begin{proof}
These are the inversions of the decompositions in Lemma \ref{lem:0sD}.
\end{proof}

\begin{corollary} 
\label{cor:D0s}    \samepage
\ifDRAFT {\rm cor:D0s}. \fi
Each of the following $8$ sequences is a basis for $V$:
\begin{align}
& \{\tau^*_i(A^*) v_0\}_{i=0}^d,
&& \{\tau^*_i(A^*) v_d\}_{i=0}^d,
&& \{\eta^*_i(A^*) v_0\}_{i=0}^d,
&& \{\eta^*_i(A^*) v_d\}_{i=0}^d,           \label{eq:tausiv0}
\\
& \{\tau_{d-i}(A) v^*_0 \}_{i=0}^d,
&& \{\tau_{d-i}(A) v^*_d\}_{i=0}^d,
&& \{\eta_{d-i}(A) v^*_0 \}_{i=0}^d,
&& \{\eta_{d-i}(A) v^*_d\}_{i=0}^d.        \label{eq:taud-ivs0}
\end{align}
\end{corollary}

\begin{proof}
By Lemma \ref{lem:D0s}.
\end{proof}

\begin{lemma}    \label{lem:0D}    \samepage
\ifDRAFT {\rm lem:0D}. \fi
For each row in the table below, consider the decomposition $\{U_i\}_{i=0}^d$ of $V$
in the first column.
For $0 \leq i \leq d$
the vector in the second column and third column is a basis for $U_i$.
\[
\begin{array}{c|c|c}
\text{\rm decomposition $\{U_i\}_{i=0}^d$} 
  & \text{\rm basis for $U_i$}  &  \text{\rm basis for $U_i$}
\\ \hline
{[0D]} & E_i v^*_0 & E_i v^*_d   \rule{0mm}{2.7ex}
\\
{[0^* D^*]} & E^*_i v_0 & E^*_i v_d   \rule{0mm}{2.5ex}
\\
{[D0]} & E_{d-i} v^*_0 & E_{d-i} v^*_d  \rule{0mm}{2.5ex}
\\
{[D^* 0^*]} & E^*_{d-i} v_0 & E^*_{d-i} v_d   \rule{0mm}{2.5ex}
\end{array}
\]
\end{lemma}

\begin{proof}
First consider the decomposition $[0D]$.
By Example \ref{exam:decompositions},  $E_i v^*_0 \in U_i$.
By \cite[Lemma 10.2]{T:survey}, $E_i v^*_0 \neq 0$.
Thus $E_i v^*_0$ is a basis for $U_i$.
Similarly $E_i v^*_d$ is a basis for $U_i$.
The proof is similar for the remaining decompositions.
\end{proof}

\begin{corollary}  
\label{cor:0D}    \samepage
\ifDRAFT {\rm cor:0D}. \fi
Each of the following $8$ sequences is a basis for $V$:
\begin{align}
& \{E_i v^*_0\}_{i=0}^d,
&& \{E_i v^*_d \}_{i=0}^d,
&& \{E_{d-i} v^*_0\}_{i=0}^d,
&& \{E_{d-i} v^*_d\}_{i=0}^d,                       \label{eq:Eivs0}
\\
& \{E^*_i v_0\}_{i=0}^d,
&& \{E^*_i v_d\}_{i=0}^d,
&& \{E^*_{d-i} v_0\}_{i=0}^d,
&& \{E^*_{d-i} v_d\}_{i=0}^d.                  \label{eq:Esiv0}
\end{align}
\end{corollary}

\begin{proof}
By Lemma \ref{lem:0D}.
\end{proof}

\begin{note}
The 24 bases \eqref{eq:tauivs0}--\eqref{eq:Esiv0} are investigated by
the second author in \cite{T:24points}.
In \cite[Theorem 11.2]{T:24points} the matrices representing $A$ and $A^*$
with respect to these 24 bases are given.
In \cite[Section 15]{T:24points}
the transition matrices between these 24 bases are given.
\end{note}

Let  $\{u_i\}_{i=0}^d$ denote a basis for $V$. Then $\{u_{d-i}\}_{i=0}^d$ is a basis for $V$,
called the {\em inversion of $\{u_i\}_{i=0}^d$}.
For each of the $24$ bases listed in \eqref{eq:tauivs0}--\eqref{eq:Esiv0},
its inversion is listed in \eqref{eq:tauivs0}--\eqref{eq:Esiv0}.

\section{Some relationship among the $24$ bases}
\label{sec:rel24}

Recall the Leonard system $\Phi$ from \eqref{eq:Phi2}.
In Lemmas \ref{lem:0sD}, \ref{lem:D0s}, \ref{lem:0D}
we gave some decompositions of $V$.
For each decomposition and  $0 \leq i \leq d$ 
we gave two bases for its $i^\text{th}$ component.
In this section  we show how these bases are related.
To do this, we consider the following inner products:

\begin{align}      
& \b{v_0, v_0},  \qquad
\b{v_d, v_d},   \qquad
\b{v^*_0, v^*_0}, \qquad
\b{v^*_d, v^*_d},                     \label{eq:first}
\\
& \b{v_0, v^*_0}, \qquad
\b{v_0, v^*_d}, \qquad
\b{v_d, v^*_0}, \qquad
\b{v_d, v^*_d}.                        \label{eq:second}
\end{align}
The above scalars are all nonzero
by \cite[Lemma 15.5]{T:qRacah} applied to the relatives of $\Phi$.

\begin{lemma}  \cite[Lemma 9.5]{NT:maps}
\label{lem:E0xis0}      \samepage
\ifDRAFT {\rm lem:E0xis0}. \fi
We have
\begin{align}
 E_0 \,v^*_0 &= \frac{\b{v_0, v^*_0}}{\b{v_0, v_0}} \, v_0,   &
 E_d \, v^*_0 &= \frac{\b{v_d, v^*_0}}{\b{v_d, v_d}} \, v_d,
                                                           \label{eq:vs0} \\
 E_0 \, v^*_d &= \frac{\b{v_0, v^*_d}}{\b{v_0, v_0}} \, v_0,   &
 E_d \, v^*_d &= \frac{\b{v_d, v^*_d}}{\b{v_d, v_d}} \, v_d,    
                                                           \label{eq:vsd} \\
 E^*_0 \, v_0 &= \frac{\b{v_0, v^*_0}}{\b{v^*_0, v^*_0}} \, v^*_0,  &
 E^*_d \, v_0 &= \frac{\b{v_0, v^*_d}}{\b{v^*_d, v^*_d}} \, v^*_d,  
                                                             \label{eq:v0} \\
 E^*_0 \, v_d &= \frac{\b{v_d, v^*_0}}{\b{v^*_0, v^*_0}} \, v^*_0, &
 E^*_d \, v_d &= \frac{\b{v_d, v^*_d}}{\b{v^*_d, v^*_d}} \, v^*_d.  
                                                              \label{eq:vd}
\end{align}
\end{lemma}

The scalars \eqref{eq:first}, \eqref{eq:second} satisfy the following relations.

\begin{lemma}  \cite[Lemma 9.7]{NT:maps}
 \label{lem:rels1}            \samepage
\ifDRAFT {\rm lem:rels1}. \fi
We have
\begin{align}     
  \frac{\b{v_0, v^*_d} \b{v_d,v^*_0}  }{ \b{v_0, v^*_0} \b{v_d,v^*_d} }
&= \frac{\vphi_1 \cdots \vphi_d}
          {\phi_1 \cdots \phi_d}.                                                     \label{eq:newrel}
\end{align}
\end{lemma}

\begin{lemma}  \cite[Corollary 8.3, Lemma 9.6]{NT:maps}
 \label{lem:rels2}            \samepage
\ifDRAFT {\rm lem:rels2}. \fi
We have
\begin{align}     
\frac{ \b{v_0,v_0} \b{v^*_0,v^*_0} } { \b{v_0,v^*_0}^2 }
  &= \frac{ \eta_d(\th_0) \eta^*_d (\th^*_0) } { \phi_1 \cdots \phi_d},   \label{eq:00s} 
\\
\frac{ \b{v_0,v_0} \b{v^*_d,v^*_d} } { \b{v_0,v^*_d}^2 }
  &= \frac{ \eta_d(\th_0) \tau^*_d(\th^*_d) } { \vphi_1 \cdots \vphi_d },   \label{eq:0ds}
\\
\frac{ \b{v_d,v_d} \b{v^*_0,v^*_0} } { \b{v_d,v^*_0}^2 }
  &= \frac{ \tau_d(\th_d) \eta^*_d(\th^*_0) } { \vphi_1 \cdots \vphi_d },    \label{eq:d0s} 
\\
\frac{ \b{v_d,v_d} \b{v^*_d,v^*_d} } { \b{v_d,v^*_d}^2 }
 &= \frac{ \tau_d(\th_d) \tau^*_d(\th^*_d) } { \phi_1 \cdots \phi_d}.     \label{eq:dds}
\end{align}
\end{lemma}

\begin{note}    \label{note:rel}    \samepage
\ifDRAFT {\rm note:rel}. \fi
By \eqref{eq:00s}--\eqref{eq:dds} the scalars \eqref{eq:second} are determined
up to sign by the scalars \eqref{eq:first} and the parameter array.
\end{note}

Our next goal is to describe 
how the bases in Lemmas \ref{lem:0sD}, \ref{lem:D0s}, \ref{lem:0D} are related.
The bases in Lemma \ref{lem:0sD} are related as follows.

\begin{lemma}  \label{lem:trans1b}    \samepage
\ifDRAFT {\rm lem:trans1b}. \fi
For $0 \leq i \leq d$,
\begin{align}
\tau^*_{d-i}(A^*)  v_0
 &= \frac{\tau^*_d(\th^*_d) } {\vphi_d \cdots \vphi_{d-i+1} } \,
     \frac{ \b{v_0,v^*_d} } { \b{v^*_d,v^*_d} }  \, \eta_i(A)  v^*_d,         \label{eq:tausd-iAsv02}
\\
\eta^*_{d-i}(A^*)  v_0
 &= \frac{\eta^*_d(\th^*_0) }  { \phi_1 \cdots \phi_i }  \,
     \frac{ \b{v_0,v^*_0} } { \b{v^*_0,v^*_0} } \,   \eta_i(A)  v^*_0,               \label{eq:etasd-iAsv02}
\\
\tau^*_{d-i}(A^*)  v_d
 &= \frac{ \tau^*_d(\th^*_d) } { \phi_d \cdots \phi_{d-i+1} }  \,
      \frac{ \b{v_d,v^*_d} } { \b{v^*_d,v^*_d} } \,  \tau_i(A)  v^*_d,   \label{eq:tausd-iAsvd2}
\\
\eta^*_{d-i}(A^*)  v_d
 &=  \frac{ \eta^*_d(\th^*_0) } { \vphi_1 \cdots \vphi_i }  \,
      \frac{ \b{v_d,v^*_0} }  { \b{v^*_0, v^*_0} }  \, \tau_i(A)  v^*_0.      \label{eq:etasd-iAsvd2}
\end{align}
\end{lemma}

\begin{proof}
We first show \eqref{eq:tausd-iAsv02}.
By \cite[Theorem 5.2]{NT:unit},
\begin{equation}
  \tau^*_{d-i}(A^*) E_0 
  = \frac{\tau^*_d(\th^*_d)} { \vphi_d \cdots \vphi_{d-i+1} } \, \eta_i(A) E^*_d E_0.       \label{eq:formula1}
\end{equation}
In this line, apply each side to $v_0$ and use $E_0 v_0 = v_0$.
Simplify the result using the equation on the right in \eqref{eq:v0}.
This gives \eqref{eq:tausd-iAsv02}.
To get the remaining equations, apply \eqref{eq:tausd-iAsv02} to 
$\Phi^\downarrow$, $\Phi^\Downarrow$, $\Phi^{\downarrow\Downarrow}$,
and use Lemma \ref{lem:Phis}.
\end{proof}

The bases in Lemma \ref{lem:D0s} are related as follows.

\begin{lemma}  \label{lem:trans1c}    \samepage
\ifDRAFT {\rm lem:trans1c}. \fi
For $0 \leq i \leq d$,
\begin{align}
\tau_{d-i}(A)  v^*_0
 &=  \frac{ \tau_d(\th_d) } { \vphi_d \cdots \vphi_{d-i+1} }   \,
      \frac{ \b{v_d,v^*_0} } {\b{v_d,v_d} }  \, \eta^*_i (A^*)  v_d,     \label{eq:taud-iAvs02}
\\
\eta_{d-i}(A)  v^*_0 
 &= \frac{ \eta_d(\th_0) } { \phi_d \cdots \phi_{d-i+1} }   \,
    \frac{ \b{v_0,v^*_0} } { \b{v_0,v_0} }  \, \eta^*_i(A^*)   v_0,      \label{eq:etad-iAvs02}
\\
\tau_{d-i}(A)  v^*_d
 &= \frac{\tau_d(\th_d) }   {\phi_1 \cdots \phi_i  } \,
      \frac{ \b{v_d,v^*_d} } { \b{v_d,v_d} }  \,  \tau^*_i(A^*)  v_d,             \label{eq:taud-iAvsd2}
\\
\eta_{d-i}(A)  v^*_d
 &=  \frac{\eta_d(\th_0) } { \vphi_1 \cdots \vphi_i }   \,
      \frac{ \b{v_0, v^*_d} } { \b{v_0,v_0} }  \, \tau^*_i(A^*)  v_0.           \label{eq:etad-iAvsd2}
\end{align}
\end{lemma}

\begin{proof}
Apply Lemma \ref{lem:trans1b} to $\Phi^*$, and use Lemma \ref{lem:Phis}.
\end{proof}

The bases in Lemma \ref{lem:0D} are related as follows.

\begin{lemma}  \label{lem:trans3b}    \samepage
\ifDRAFT {\rm lem:trans3b}. \fi
For $0 \leq i \leq d$,
\begin{align}
E^*_i  v_d
&= \frac{\phi_1 \cdots \phi_i } {\vphi_1 \cdots \vphi_i } \,
     \frac{ \b{v_d,v^*_0} } { \b{v_0,v^*_0} } \, E^*_i  v_0,                     \label{eq:Esivd2}
\\
E^*_{d-i}  v_d
 &= \frac{\vphi_d \cdots \vphi_{d-i+1} } {\phi_d \cdots \phi_{d-i+1} } \,
      \frac{ \b{v_d,v^*_d} } { \b{v_0,v^*_d} }     \, E^*_{d-i}  v_0,           \label{eq:Esd-ivd2}
\\
E_i  v^*_d
&= \frac{\phi_d \cdots \phi_{d-i+1} } {\vphi_1 \cdots \vphi_i } \,
     \frac{ \b{v_0,v^*_d} } { \b{v_0,v^*_0} }   \, E_i  v^*_0,   \label{eq:Eivsd2}
\\
E_{d-i}  v^*_d
 &= \frac{\vphi_d \cdots \vphi_{d-i+1} }   {\phi_1 \cdots \phi_i } \,
      \frac{ \b{v_d,v^*_d} } { \b{v_d,v^*_0} }        \, E_{d-i}  v^*_0.        \label{eq:Ed-ivsd2}
\end{align}
\end{lemma}

\begin{proof}
We first show \eqref{eq:Esivd2}.
In the equation on the right in  \eqref{eq:v0}, multiply each side on the left by $E^*_i E_d$.
Simplify the result using the equation on the right in \eqref{eq:vsd}. This gives
\begin{equation}
 E^*_i E_d E^*_d v_0 =
  \frac{\b{v_0, v^*_d} \b{v_d,v^*_d} } { \b{v_d, v_d} \b{v^*_d, v^*_d} } \, E^*_i v_d.  \label{eq:Esivdaux1}
\end{equation}
By \cite[Lemma 7.2]{NT:maps},
\begin{equation}
 E_0 E^*_d E_d E^*_i  =
   \frac{\vphi_1 \cdots \vphi_d } { \tau_d(\th_d) \tau^*_d (\th^*_d) } \,
   \frac{\phi_1 \cdots \phi_i } { \vphi_1 \cdots \vphi_i } \, E_0 E^*_i.                \label{eq:formula2pre}
\end{equation}
Applying $\dagger$ to \eqref{eq:formula2pre} we obtain
\begin{equation}
  E^*_i E_d E^*_d E_0 =
   \frac{\vphi_1 \cdots \vphi_d } { \tau_d(\th_d) \tau^*_d (\th^*_d) } \,
   \frac{\phi_1 \cdots \phi_i } { \vphi_1 \cdots \vphi_i } \, E^*_i E_0.             \label{eq:formula2}
\end{equation}
In this line, apply each side to $v_0$, and use $E_0 v_0 = v_0$.
Comparing the result with \eqref{eq:Esivdaux1} we find that $E^*_i v_d$ is equal to
\begin{equation}
    \frac{\vphi_1 \cdots \vphi_d} {\tau_d(\th_d) \tau^*_d(\th^*_d) } \,
    \frac{\b{v_d,v_d} \b{v^*_d,v^*_d} } { \b{v_0,v^*_d} \b{v_d,v^*_d} }             \label{eq:Esivdaux2}
\end{equation}
times
\[
     \frac{\phi_1 \cdots \phi_i} {\vphi_1 \cdots \vphi_i } \, E^*_i v_0.
\]
By \eqref{eq:newrel} and  \eqref{eq:dds},
the line \eqref{eq:Esivdaux2} is equal to
\[
     \frac{ \b{v_d,v^*_0} } { \b{v_0,v^*_0} }.
\]
By these comments we obtain \eqref{eq:Esivd2}.
To get  \eqref{eq:Esd-ivd2},
replace $i$ with $d-i$ in \eqref{eq:Esivd2} and use \eqref{eq:newrel}.
To get \eqref{eq:Eivsd2}, apply \eqref{eq:Esivd2} to $\Phi^*$, and use Lemma \ref{lem:Phis}.
The line \eqref{eq:Ed-ivsd2} is similarly obtained by applying \eqref{eq:Esd-ivd2} to $\Phi^*$.
\end{proof}

\section{The action of $T$ on the 24 bases}
\label{sec:Taction}

Recall the Leonard system $\Phi$ from \eqref{eq:Phi2} and 
the element $T$  from Definition \ref{def:T}.
Consider the $24$ bases from \eqref{eq:tauivs0}--\eqref{eq:Esiv0}.
In this section we describe how $T$ acts on these bases,
under the assumption that $\Phi$ is self-dual.

\begin{lemma}     \label{lem:actT}    \samepage
\ifDRAFT {\rm lem:actT}. \fi
Assume that $\Phi$ is self-dual.
Then
\begin{align*}
 T v_0 &= \alpha v^*_0, \qquad\qquad \;
 T v_d = \beta v^*_d,  
\\
 T v^*_0 &= \alpha^* v_0, \qquad\qquad 
 T v^*_d = \beta^* v_d,  
\end{align*}
where
\begin{align*}
\alpha &= \frac{\vphi_1 \cdots \vphi_d}{\tau_d(\th_d)} \, 
              \frac{ \b{v_0,v^*_0} } { \b{v^*_0,v^*_0} },
&
\beta &=  \frac{\vphi_1 \cdots \vphi_d} { \eta_d(\th_0) } \, 
             \frac{ \b{v_d,v^*_d} } { \b{v^*_d, v^*_d} },
\\
\alpha^* &= \frac{\vphi_1 \cdots \vphi_d} { \tau_d(\th_d)} \,
                 \frac{ \b{v_0,v^*_0} } { \b{v_0,v_0} },        
&
\beta^* &= \frac{\vphi_1 \cdots \vphi_d} {\eta_d(\th_0)}  \,
             \frac{ \b{v_d, v^*_d} } { \b{v_d,v_d} }.
\end{align*}
\end{lemma}

\begin{proof}
By Theorem \ref{thm:main} we have $T=T^*$.
By this and \eqref{eq:TsE0},
\[
   T E_0 = \frac{\vphi_1 \cdots \vphi_d} {\tau_d(\th_d)} \, E^*_0 E_0.
\]
In this line, apply each side to $v_0$.
Simplify the result using $E_0 v_0 = v_0$ to get
\[
   T v_0 = \frac{\vphi_1 \cdots \vphi_d} {\tau_d(\th_d)} \, E^*_0 v_0.
\]
In this line, eliminate $E^*_0 v_0$ using the equation on the left in \eqref{eq:v0} to
get $T v_0 = \alpha v^*_0$.
The remaining equations are obtained in a similar way.
\end{proof}

\begin{proposition}    \label{prop:TEivs0}    \samepage
\ifDRAFT {\rm prop:TEivs0}. \fi
Assume that $\Phi$ is self-dual.
Then for  $0 \leq i \leq d$,
\begin{align*}
 T E^*_i v_0 &= \alpha E_i v^*_0,  &
 T \tau^*_i (A^*) v_0 &= \alpha  \tau_i (A) v^*_0,  &
 T \eta^*_i (A^*) v_0 &= \alpha  \eta_i (A) v^*_0,         
\\
 T E^*_i v_d &= \beta E_i v^*_d, &
 T \tau^*_i (A^*) v_d &= \beta \tau_i (A) v^*_d,  &
 T \eta^*_i (A^*) v_d &= \beta \eta_i (A) v^*_d,   
\\
 T E_i v^*_0 &= \alpha^* E^*_i v_0, &
 T \tau_i (A) v^*_0 &= \alpha^* \tau^*_i (A^*) v_0, &
 T \eta_i (A) v^*_0 &= \alpha^* \eta^*_i (A^*) v_0,  
\\
 T E_i v^*_d &= \beta^* E^*_i v_d, &
 T \tau_i (A) v^*_d &= \beta^* \tau^*_i (A^*) v_d, &
 T \eta_i (A) v^*_d &= \beta^* \eta^*_i (A^*) v_d,    
\end{align*}
where $\alpha$, $\beta$, $\alpha^*$, $\beta^*$ are from Lemma \ref{lem:actT}.
\end{proposition}

\begin{proof}
By \eqref{eq:EiT} we have  $T E^*_i T^{-1} = E_i$.
By Lemma \ref{lem:actT}, $T v_0 = \alpha v^*_0$.
By these comments
\[
  T E^*_i v_0 = T E^*_i T^{-1} T v_0  = \alpha E_i v^*_0.
\]
We have shown that $T E^*_i v_0 = \alpha E_i v^*_0$.
The remaining equations are obtained in a similar way.
\end{proof}

To motivate the next result we make some comments.
Consider the following bases for $V$:
\begin{equation}
  \{\eta^*_i (A^*) v_0 \}_{i=0}^d,  \qquad                  \label{eq:4bases}
  \{\eta_i(A) v^*_0 \}_{i=0}^d,   \qquad
 \{\tau^*_i(A^*) v_d \}_{i=0}^d,  \qquad
 \{\tau_i(A) v^*_d\}_{i=0}^d. 
\end{equation}
By \cite[Theorem 11.2]{T:24points},
with respect to these bases the matrices representing $A$ and $A^*$
are as follows.
\[
\begin{array}{c|ccc}
\text{\rm basis}  & &  \text{\rm matrix representing $A$} 
           & \text{\rm matrix representing $A^*$}
\\ \hline
\{\eta^*_i(A^*) v_0 \}_{i=0}^d       \rule{0mm}{12ex}
& &
  \begin{pmatrix}
    \th_0 & \phi_d & & & & \text{\bf 0}  \\
            & \th_1 & \cdot \\
            &              &  \cdot & \cdot  \\
            &              &           & \cdot &  \phi_2 \\
            &              &           &          & \th_{d-1} & \phi_1  \\
      \text{\bf 0}     &               &           &          &          & \th_d
  \end{pmatrix}
&
 \begin{pmatrix}
   \th^*_d & & & & & \text{\bf 0}   \\
       1  & \th^*_{d-1} \\
           & 1     & \cdot  \\
           &        & \cdot & \cdot \\
           &        &          & \cdot & \th^*_1  \\
     \text{\bf 0}   &       &          &          &  1   & \th^*_0
  \end{pmatrix}
\\
\{\eta_i(A) v^*_0\}_{i=0}^d         \rule{0mm}{12ex}
& &
\begin{pmatrix}
   \th_d & & & & & \text{\bf 0}   \\
       1  & \th_{d-1} \\
           & 1     & \cdot  \\
           &        & \cdot & \cdot \\
           &        &          & \cdot & \th_1  \\
     \text{\bf 0}   &       &          &          &  1   & \th_0
  \end{pmatrix}
&
  \begin{pmatrix}
     \th^*_0 & \phi_1 & & & & \text{\bf 0}   \\
                & \th^*_1 & \phi_2  \\
                &            & \cdot & \cdot  \\
                &             &         & \cdot & \cdot \\
               &             &           &         & \th^*_{d-1} & \phi_d \\
    \text{\bf 0} &       &            &          &                & \th^*_d
  \end{pmatrix}
\\
\{\tau^*_i(A^*) v_d \}_{i=0}^d       \rule{0mm}{12ex}
& &
  \begin{pmatrix}
    \th_d & \phi_1 & & & & \text{\bf 0}  \\
            & \th_{d-1} & \phi_2 \\
            &              &  \cdot & \cdot  \\
            &              &           & \cdot &  \cdot \\
            &              &           &          & \th_1 & \phi_d  \\
      \text{\bf 0}     &               &           &          &          & \th_0
  \end{pmatrix}
&
 \begin{pmatrix}
   \th^*_0 & & & & & \text{\bf 0}   \\
       1  & \th^*_1 \\
           & 1     & \cdot  \\
           &        & \cdot & \cdot \\
           &        &          & \cdot & \th^*_{d-1}  \\
     \text{\bf 0}   &       &          &          &  1   & \th^*_d
  \end{pmatrix}
\\
\{\tau_i(A) v^*_d\}_{i=0}^d           \rule{0mm}{12ex}
& &
 \begin{pmatrix}
   \th_0 & & & & & \text{\bf 0}   \\
       1  & \th_1 \\
           & 1     & \cdot  \\
           &        & \cdot & \cdot \\
           &        &          & \cdot & \th_{d-1}  \\
     \text{\bf 0}   &       &          &          &  1   & \th_d
  \end{pmatrix}
&
  \begin{pmatrix}
     \th^*_d & \phi_d & & & & \text{\bf 0}   \\
                & \th^*_{d-1} & \cdot  \\
                &            & \cdot & \cdot  \\
                &             &         & \cdot & \phi_2 \\
               &             &           &         & \th^*_1 & \phi_1 \\
    \text{\bf 0} &       &            &          &                & \th^*_0
  \end{pmatrix}
\end{array}
\]

\begin{theorem}    \label{thm:matrixT}     \samepage
\ifDRAFT {\rm thm:matrixT}. \fi
Assume that $\Phi$ is self-dual.
Then with respect to each basis  \eqref{eq:4bases}
the matrix representing $T$ is
\begin{equation}
\frac{\vphi_1 \cdots \vphi_d} {\tau_d(\th_d) \eta_d(\th_0) } \,
  \begin{pmatrix}
    \text{\bf 0} & & & & & \phi_1 \cdots \phi_d       \\
     & & & & \cdot \\
     & & & \cdot \\
     & & \phi_1 \phi_2 \\
     & \phi_1 \\
    1 & & & & & \qquad\;\;  \text{\bf 0}
  \end{pmatrix}.                                               \label{eq:mat}
\end{equation}
\end{theorem}

\begin{proof}
First consider the basis $\{\eta^*_i (A^*) v_0\}_{i=0}^d$.
By Proposition \ref{prop:TEivs0},
\begin{align*}
   T \eta^*_i (A^*) v_0  &= 
    \frac{\vphi_1 \cdots \vphi_d} { \tau_d(\th_d) } \,
    \frac{ \b{v_0,v^*_0} } { \b{v^*_0, v^*_0} } \,   \eta_i (A) v^*_0  && (0 \leq i \leq d).
\end{align*}
By \eqref{eq:thths} and \eqref{eq:etasd-iAsv02},
\begin{align*}
   \eta_i (A) v^*_0 &=
    \frac{\phi_1 \cdots \phi_i} { \eta_d(\th_0) } \,
    \frac{ \b{v^*_0, v^*_0} } { \b{v_0, v^*_0} } \,  \eta^*_{d-i} (A^*) v_0  && (0 \leq i \leq d).
\end{align*}
By these comments, 
\begin{align}
T \eta^*_i (A^*) v_0 &=
 \frac{\vphi_1 \cdots \vphi_d} { \tau_d(\th_d) \eta_d(\th_0) } \,
  \phi_1 \cdots \phi_i \, \eta^*_{d-i} (A^*) v_0          && (0 \leq i \leq d).      \label{eq:TetasiAsv0}
\end{align}
Thus the matrix \eqref{eq:mat} represents $T$ with respect to  $\{\eta^*_i (A^*) v_0\}_{i=0}^d$.

Next consider the basis  $\{\tau^*_i (A^*) v_d\}_{i=0}^d$.
In a similar way as above using \eqref{eq:phi} and \eqref{eq:tausd-iAsvd2},
we obtain 
\begin{align}
T \tau^*_i (A^*) v_d &=
\frac{\vphi_1 \cdots \vphi_d} { \tau_d(\th_d) \eta_d(\th_0) } \,
  \phi_1 \cdots \phi_i \, \tau^*_{d-i} (A^*) v_d   && (0 \leq i \leq d).             \label{eq:TtausiAsvd}
\end{align}
Thus the matrix \eqref{eq:mat} represents $T$ with respect to  $\{\tau^*_i (A^*) v_d\}_{i=0}^d$.

Next consider the basis $\{\eta_i(A) v^*_0\}_{i=0}^d$.
Apply \eqref{eq:TetasiAsv0} to $\Phi^*$, and use Lemma \ref{lem:Phis}.
This gives
\begin{align*}
 T^* \eta_i (A) v^*_0 &=
  \frac{\vphi_1 \cdots \vphi_d } { \tau^*_d (\th^*_d) \eta^*_d (\th^*_0) } \,
  \phi_d \cdots \phi_{d-i+1} \, \eta_{d-i}(A) v^*_0         && (0 \leq i \leq d).
\end{align*}
We have $T^*=T$ by Theorem \ref{thm:main}.
By this and \eqref{eq:thths}, \eqref{eq:phi}, 
the above line becomes
\begin{align*}
T \eta_i (A) v^*_0 &=
 \frac{\vphi_1 \cdots \vphi_d} { \tau_d(\th_d) \eta_d(\th_0) } \,
  \phi_1 \cdots \phi_i \, \eta_{d-i} (A) v^*_0      &&  (0 \leq i \leq d).
\end{align*}
Thus the matrix \eqref{eq:mat} represents $T$ with respect to $\{\eta_i(A) v^*_0\}_{i=0}^d$.

Next consider the basis $\{\tau_i (A) v^*_d\}_{i=0}^d$.
Apply \eqref{eq:TtausiAsvd} to $\Phi^*$, and use Lemma \ref{lem:Phis}.
Simplify the result in a similar way as above to get
\begin{align*}
T \tau_i (A) v^*_d &=
\frac{\vphi_1 \cdots \vphi_d} { \tau_d(\th_d) \eta_d(\th_0) } \,
  \phi_1 \cdots \phi_i \, \tau_{d-i} (A) v^*_d    && (0 \leq i \leq d).   
\end{align*}
Thus the matrix \eqref{eq:mat} represents $T$ with respect to $\{\tau_i (A) v^*_d\}_{i=0}^d$.
\end{proof}


\bigskip


{

\small

\begin{thebibliography}{10}


\bibitem{C:thin}
B. Curtin,
Distance-regular graps which supports a spin model are thin,
Discrete Math.\ 197/198 (1999), 205--216.

\bibitem{Cur}
B. Curtin,
Modular Leonard triples,  
Linear algebra Appl.\  424 (2007), 510--539. 


\bibitem{CN:hyper}
B. Curtin, K. Nomura,
Spin models and strongly hyper-self-dual Bose-Mesner algebras,
J. Algebraic Combin.\ 13 (2004), 173--186.

\bibitem{Go}
J.T. Go, The Terwilliger algebra of the hypercube,
European J. Comb.\ 23 (2002), 399--429.

\bibitem{IRT}
T. Ito, H. Rosengren, P. Terwilliger,
Evaluation modules for the $q$-tetrahedron algebra,
Linear Algebra Appl.\ 451 (2014), 107--168.


\bibitem{NT:unit}
K. Nomura, P. Terwilliger,
Matrix units associated with the split basis of a Leonard pair,
Linear Algebra Appl.\ 418 (2006), 775--787.

\bibitem{NT:affine}
K. Nomura, P. Terwilliger,
Affine transformations of a Leonard pair,
Electron. J. Linear Algebra 16 (2007), 389-418.

\bibitem{NT:switch}
K. Nomura, P. Terwilliger,
The switching element for a Leonard pair,
Linear Algebra Appl.\  428 (2008), 1083--1108.

\bibitem{NT:maps}
K. Nomura, P. Terwilliger,
Transition maps between the $24$ bases for a Leonard pair,
Linear Algebra Appl.\ 431 (2009), 571--593.

\bibitem{NT:Krawt}
K. Nomura, P. Terwilliger,
Krawtchouk polynomials, the Lie algebra $\mathfrak{sl}_2$, and Leonard pairs,
Linear Algebra Appl.\ 437 (2012), 345--375.

\bibitem{NT:bipartite}
K. Nomura, P. Terwilliger,
Totally bipartite tridiagonal pairs,
preprint, 
arXiv:1711.00332.

\bibitem{Rot}
J. Rotman,
Advanced Modern Algebra,
2nd edition,
AMS, Providence, RI, 2010.

\bibitem{T:Leonard}
P.\ Terwilliger,
Two linear transformations each tridiagonal with respect to an eigenbasis
of the other, 
Linear Algebra Appl.\ 330 (2001), 149--203.

\bibitem{T:24points}
P. Terwilliger, 
Leonard pairs from 24 points of view, 
Rocky Mountain J. Math.\ 32 (2002), 827--887.


\bibitem{T:qRacah}
P. Terwilliger,
Leonard pairs and the $q$-Racah polynomials,
Linear Algebra Appl.\ 387 (2004), 235--276.

\bibitem{T:survey}
P. Terwilliger,
An algebraic approach to the Askey scheme of orthogonal polynomials,
Orthogonal polynomials and special functions, 
Lecture Notes in Math., 1883, 
Springer, Berlin, 2006, pp. 255--330,
arXiv:math/0408390.


\end{thebibliography}


}
\end{document}